\documentclass{article}
\usepackage{hyperref}
\usepackage{amssymb,amsmath,amsthm}
\usepackage{enumerate}
\usepackage{color, ulem}

\newtheorem{theorem}{Theorem}[section]
\newtheorem{corollary}[theorem]{Corollary}
\newtheorem{lemma}[theorem]{Lemma}
\newtheorem{proposition}[theorem]{Proposition}

\theoremstyle{definition}
\newtheorem{definition}[theorem]{Definition}
\newtheorem{remark}[theorem]{Remark}

\newtheorem{note}[theorem]{Note}

\numberwithin{equation}{section}  

\begin{document}

\title{Harmonic Hadamard manifolds and Gauss hypergeometric differential equations}

\author{Mitsuhiro Itoh\footnote{Institute of Mathematics, University of Tsukuba, 1-1-1 Tennodai, Tsukuba-shi, Ibaraki 305-8577, JAPAN;
e-mail : itohm@math.tsukuba.ac.jp}
and
Hiroyasu Satoh\footnote{Department of Human Science and Common Education, Faculty of Engineering, Nippon Institute of Technology, 4-1 Gakuendai, Miyashiro-machi, Minamisaitama-gun, Saitama 345-8501 JAPAN; e-mail : hiroyasu@nit.ac.jp}}


\maketitle

\begin{abstract}
A new class of harmonic Hadamard manifolds,  those spaces called of hypergeometric type, is defined in terms of Gauss hypergeometric equations.
Spherical Fourier transform defined on a harmonic Hadamard manifold of hypergeometric type admits an inversion formula.  A characterization of harmonic Hadamard manifold being of hypergeometric type is obtained with respect to volume density.\bigskip
\end{abstract}

\section{Introduction}\label{intro} 
A Riemann manifold is called harmonic if it admits a solution of Laplace equation $\Delta f = -\sum_{i,j}g^{ij}\nabla_i\nabla_j f = 0$ depending only on the distance. A euclidean space is typically  a harmonic manifold. In fact, on a euclidean space  ${\Bbb R}^n$ there exists a solution of $\Delta f = 0$ written as $f(x) = \vert x \vert^{2-n}$ if $n > 2$,\, $f(x) = \log (\vert x \vert)$ if $n = 2$.
In 1930 H.S. Ruse considered the following; on any Riemannian manifold there exists a solution of $\Delta f = 0$ described as a function of the distance function.
While his attempt has failed, his idea yielded the notion of harmonic manifold, the study of which is an important subject geometers are intensively interested in.
The Laplace equation $\Delta f = 0$ is very important in analyzing natural, especially physical phenomena.
Therefore, a theory of harmonic manifold arising in differential geometry crucially relates to harmonic analysis which analyzes natural phenomena. 

The study of harmonic manifolds started with curvature conditions, e.g., Ledger's formulae. Using these conditions harmonic manifolds of dimension $\leq 4$  are proved to be flat or rank one, locally symmetric (\cite{Lich}). Notice that a harmonic manifold is always Einstein.  
  
  The second stage of study of harmonic manifolds has begun with Lichnerowicz conjecture which asserts that a harmonic manifold must be flat or locally symmetric, rank one. Still, we do not have a complete answer to his conjecture, even there is  partially affirmative progress for a compact case. See \cite{Sz}, \cite{BCG} for this. Counter examples to the conjecture appeared in 1980's by work of Damek and Ricci (\cite{DR}), as non-symmetric non-compact, harmonic manifolds, called Damek-Ricci spaces to non-compact version of Lichnerowicz conjecture. They are two step solvable Lie groups with a left invariant metric.  Notice that a  Damek-Ricci space is  symmetric if and only if it is a rank one symmetric space of noncompact type. For details of Damek-Ricci spaces we refer to \cite{BTV}. 
      
    Radial functions, like the volume density and the mean curvature of geodesic spheres, are essentially significant for analyzing a harmonic manifold. On a harmonic manifold convolution of two radial functions is also radial, and moreover there exists a radial eigenfunction of Laplace-Beltrami operator for each eigenvalue, as illustrated by Szab\'o in \cite{Sz}.

    Let $(X,g)$ be a harmonic Hadamard manifold of volume entropy, volume exponential growth rate $\rho(X,g) = Q > 0$.  Here, a simply connected, complete Riemannian manifold of nonpositive sectional curvature is called an Hadamard manifold.
    Then,  there exists for each $\lambda\in {\Bbb C}$ a radial function $\varphi_{\lambda}(r)$ of $\varphi_{\lambda}(0) = 1$ which satisfies the eigen-Laplace equation $\Delta h = (\frac{Q^2}{4}+\lambda^2)h$, $h = \varphi_{\lambda}(r)$. We call $\varphi_{\lambda}$ a spherical function on $X$. Then, using the $\varphi_{\lambda}$'s, spherical Fourier transform is defined for a radial function $f = f(r)$ of compact support as $f \mapsto {\mathcal H}f(\lambda)$, $\lambda\in {\Bbb C}$;
    \begin{equation}\label{sphfourier}
    {\mathcal H} f(\lambda) := \int_X f(x) \varphi_{\lambda}(x) dv_g = \omega_{n-1}\, \int_0^{\infty} f(r) \varphi_{\lambda}(r) \Theta(r) dr.
    \end{equation}Here, a function $f(x)$, $x\in X$ is identified with $f(r(x))$, $\omega_{n-1}$ is the volume of the unit $(n-1)$-sphere and $\Theta(r)$ denotes volume density of a geodesic sphere of radius $r$.
    
A theory of spherical Fourier transform on symmetric spaces  has been initiated and intensively studied by G. Helgason, especially on  rank one symmetric spaces of non-compact type, which are typical non-compact harmonic manifolds (\cite{H}).
The inversion formula, a Plancherel theorem\, and a Paley-Wiener theorem are established like the euclidean spherical Fourier transform.
    
    Theory of spherical Fourier transform has been developed over a Damek-Ricci space by \cite{ADY, R, DR2, Ri}.  Anker et al. pointed out in \cite{ADY} that a spherical Fourier transform over a Damek-Ricci space is settled in a frame of Jacobi operator and is represented by the aid of Gauss hypergeometric functions. 
We refer to \cite{PS} also for spherical Fourier transform on a non-compact harmonic manifold. In \cite[Proposition 3.10]{PS} Peyerimhoff and Samiou prove that $\displaystyle{{\mathcal H} F = \widehat{{\mathcal A} F}}^{cl}$, where $F$ is any radial distribution on $X$ with compact support and ${\mathcal{A}}$ is the Abel transform. Here the hat $\widehat{(\,\cdot\,)}^{cl}$ is the euclidean classical Fourier transform. So, the spherical Fourier transform ${\mathcal H}$ on $X$ turns out to be an isomorphism in the level of distributions.
    
    The aim of this paper  is to define a subclass of non-compact harmonic manifolds, called harmonic Hadamard manifolds of hypergeometric type and to develop spherical Fourier transform on such a harmonic manifold.

\begin{definition}\label{def}
A harmonic Hadamard manifold of positive volume entropy is said to be of hypergeometric type, if a spherical function and hence every spherical function satisfies a Gauss hypergeometric differential equation under a certain transform of variable; $\displaystyle{z = - \sinh^2 \frac{r}{2}}$.
\end{definition}
 
 For precise definition of hypergeometric type see Definition \ref{defhypergeometric}.  A Damek-Ricci space is of hypergeometric type as indicated in Proposition \ref{DR}.  
 \begin{theorem}\label{maintheorem}
 Let $(X^n,g)$ be a harmonic Hadamard manifold of volume entropy $\rho(X^n,g) = Q > 0$. If $(X,g)$ is of hypergeometric type, then
\begin{enumerate}
\item any spherical function $\varphi_{\lambda}(r)$,\, $\lambda\in \Bbb{C}$ on $(X^n,g)$ can be described as a Gauss hypergeometric function 
 \begin{equation}\label{sphericalgausshyp}
 \varphi_{\lambda}(r) = F\left(\frac{Q}{2}- i \lambda, \frac{Q}{2} + i \lambda, \frac{n}{2}; - \sinh^2 \frac{r}{2}\right)
 \end{equation} 
\item the volume density $\Theta(r)$ of a geodesic sphere $S(q;r)$ has the form 
 \begin{equation}\label{volumedensity}\Theta(r) = k_g
 \, \sinh^{n-1} \frac{r}{2}\, \cosh^{(2Q-(n-1))} \frac{r}{2}, 
 \end{equation}where 
 \begin{equation}\label{k}
 k_g = - \frac{ 2^n}{3 Q - (n-1)}\ {\rm Ric}_g
 \end{equation}
 and ${\rm Ric}_g$ is the Ricci curvature of $(X,g)$.
\end{enumerate}
\end{theorem}
Here 
 \begin{equation}\label{definitionhypergeometric}\displaystyle F\left(a,b,c ; z\right) = \sum_{n=0}^{\infty} \frac{(a)_n (b)_n}{(c)_n} \frac{z^n}{n!} = 1 + \frac{ab}{c} z + \frac{a(a+1)b(b+1)}{c(c+1)} \frac{z^2}{2!} + \cdots
 \end{equation}is the Gauss hypergeometric function, where
\begin{equation}\displaystyle (a)_n := \frac{\Gamma(a +n)}{\Gamma(a)} = a(a+1)\cdots (a + n - 1),\ (a)_0 = 1.
\end{equation} 
Basic materials of the hypergeometric functions needed in this paper will be summarized in section \ref{hypergeometricfunctions}.

Theorem \ref{maintheorem} is verified by using asymptotical formulae of the mean curvature of geodesic spheres, as indicated in Lemma \ref{meancurvature}.

  From (ii) of Theorem \ref{maintheorem} a harmonic Hadamard manifold being of hypergeometric type is characterized as follows.
  
  \begin{theorem}\label{character}
Let $(X,g)$ be a harmonic Hadamard manifold of volume entropy $\rho(X,g) = Q>0$. Then, $(X,g)$ is of hypergeometric type if and only if there exist constants $c_1$, $c_2$ satisfying $c_1 >0$ and $c_1+c_2 > 0$ such that
either the mean curvature $\sigma(r)$  or the volume density $\Theta(r)$ of a geodesic sphere $S(q;r)$ fulfills
\begin{equation} 
\sigma(r) = c_1\,\coth \frac{r}{2} + c_2\, \tanh \frac{r}{2}
\end{equation}for any $r > 0$, or
\begin{equation}
\Theta(r) = k \sinh^{2c_1} \frac{r}{2}\, \cosh^{2c_2} \frac{r}{2} = k \cosh^{2(c_1+c_2)} \frac{r}{2}\, \tanh^{2c_1} \frac{r}{2}
\end{equation}for a constant $k > 0$ and all $r > 0$.
\end{theorem}

\begin{remark}
The hypergeometric type can be defined for a class of non-compact complete harmonic manifolds which admits spherical functions including harmonic Hadamard manifolds of volume entropy $Q >0$. However, as Theorem \ref{character} shows, the hypergeometric type can be characterized in terms of the volume density function provided a harmonic manifold is Hadamard.
\end{remark}

\begin{remark}
Nikolayevsky proved that if $(X, g)$ is a harmonic manifold, then the density function $\Theta(r)$ is an exponential polynomial (\cite[Theorem 2]{N}, see also \cite[\S 4]{K2016}).
Theorem \ref{character} characterizes harmonic manifolds whose density functions are special type of exponential polynomials.
\end{remark}

From Theorem \ref{maintheorem}, for a harmonic Hadamard manifold $(X,g)$ with $\rho = Q > 0$ and of hypergeometric type there exist $0 < a \leq b$ such that 
  \begin{equation*}
     a\, e^{Q r} \leq \Theta(r) \leq b\, e^{Qr}
  \end{equation*}
  for $r \geq 1$. So, this means that $(X,g)$ has purely exponential volume growth. Therefore, by the aid of Theorem of Knieper \cite{K} we have  
 \begin{corollary}\label{gromov}
 A harmonic Hadamard manifold $(X,g)$ of volume entropy $\rho(X,g) = Q>0$ and of hypergeometric type has
 purely exponential volume growth and consequently
 \begin{enumerate}
 \item the geodesic flow on the tangent sphere bundle over $X$ is Anosov,
 \item $(X,g)$ is Gromov hyperbolic and
 \item every geodesic $\gamma$ on $X$ is rank one, that is,  the velocity vector field $\gamma'$ is the only parallel Jacobi field along $\gamma$.
 \end{enumerate}
 \end{corollary}
 
 The following is a direct application of Theorem \ref{maintheorem}.
 
 \begin{theorem}[\cite{IKPS}]\label{kaehler}
Let $(X,g,J)$ be a K\"ahler Hadamard manifold of $\dim_{\mathbb{C}}
X = m \geq 2$. Assume that $(X,g,J)$ is of volume entropy
{{$\rho(X,g) = m$}} and of Ricci curvature ${\rm Ric}_g = -
\frac{1}{2} (m+1)$ \hspace{0.5mm} and that $(X,g,J)$ is a harmonic
manifold of hypergeometric type. Then, $(X,g,J)$ is
biholomorphically isometric to the complex hyperbolic space
$\mathbb{C}H^m$ of holomorphic sectional curvature $ - 1$.
\end{theorem}

The spherical Fourier transform defined over a harmonic Hadamard manifold $(X,g)$ in (\ref{sphfourier}) admits an inversion formula as exhibited as follows, provided $(X,g)$ is of positive volume entropy and of  hypergeometric type. 
  \begin{theorem}\label{inversionthm}
  { Let $(X,g)$ be a harmonic Hadamard manifold of positive volume entropy and of  hypergeometric type.
  Let $f = f(r)$ be a smooth radial function on $(X,g)$ of compact support. Then
 \begin{equation}\label{inversionformula}
   f(r) = {d}\ \int_{-\infty}^{\infty} \frac{d \lambda}{\vert{\bf c}(\lambda)\vert^2} {\mathcal H}f(\lambda) {\overline\varphi}_{\lambda}(r),
 \end{equation}
 where the spherical functions $\varphi_{\lambda}(r)$ are real valued. 
Here, the constant $d$ is defined as
\begin{eqnarray}
d = \frac{2^{2Q}}{4 \pi\hspace{0.5mm}\omega_{n-1}\hspace{0.5mm}k_g}  = \frac{2^{2Q-3}\Gamma(\frac{n}{2})}{\pi^{n/2+1}\hspace{0.5mm}k_g} 
\end{eqnarray}
and ${\bf c}(\lambda) $ is a Harish-Chandra $c$-function with respect to $\lambda\in {\Bbb R}$, given in \cite[pp.648]{ADY} as
 \begin{equation}\label{harishChandrac}
 {\bf c}(\lambda) =  2^{(Q-2i\lambda)}\ \frac{\Gamma( \frac{n}{2}) \Gamma(2 i\lambda)}{\Gamma(\frac{n}{2}- \frac{Q}{2} + i\lambda) \Gamma(\frac{Q}{2} + i \lambda)}.
  \end{equation}
 }
 \end{theorem}
   \begin{remark}\label{damekriccic0}
   {
   For a Damek-Ricci space $S$ the transform (\ref{sphfourier}) is written as \cite[(2.8)]{ADY} and the inversion formula as in \cite[(2.10)]{ADY}
   \begin{equation*}f(r) = c_0 \int_0^{\infty} {\mathcal H}f(\lambda) \varphi_{\lambda}(r)\frac{ d \lambda}{\vert {\bf c}(\lambda)\vert^2} 
   = \frac{c_0}{2} \int_{-\infty}^{\infty} {\mathcal H}f(\lambda) \varphi_{\lambda}(r)\frac{ d \lambda}{\vert {\bf c}(\lambda)\vert^2},
\end{equation*}
where $c_0 = 2^{k-2} \pi^{-(n/2 + 1)} \Gamma(n/2)$, $k = \dim {\mathfrak z}$ and
\begin{equation}\label{damekriccicfunction}
 {\bf c}(\lambda) = \, 2^{(Q-2i\lambda)}\, \frac{\Gamma(\frac{n}{2})\Gamma(2i\lambda)}{\Gamma(\frac{Q}{2} + i\lambda) \Gamma(\frac{m}{4} + \frac{1}{2} + i\lambda)},
\end{equation}
where $Q = \frac{m}{2} + k$ is the volume entropy of $S$. 
The function ${\bf c}(\lambda)$ of (\ref{damekriccicfunction}) coincides with our $c$-function ${\bf c}(\lambda)$.
Remark \ref{dc0} indicates equality $d = c_0/2$ for the constants $d$ and $c_0$. Thus it is concluded that the spherical Fourier transform on $(X,g)$ turns out to be a natural generalization of the spherical Fourier transform on a Damek-Ricci space. 
See also \cite[Thm\, 15]{R}.  }
 \end{remark}
    A proof of Theorem \ref{inversionthm} is mainly based on the argument of G\"otze (\cite{G}).  We might use a theory of Jacobi transform, a generalization of spherical Fourier transform. See for this \cite{Ko} in which the inversion formula, a Plancherel theorem and a Paley-Wiener theorem are given in a framework of Jacobi transform.
We will make use of the argument of Koornwinder fully in future. However, we focus in this article on the inversion formula over a harmonic manifold of hypergeometric type, using Gauss hypergeometric functions under  Green's formula, familiar and well-known, adopted in \cite{G}.
 
\section{Preliminaries}\label{prelim}
We begin with some basic preliminaries for Hadamard manifolds. We refer to \cite{ISS} for details. Let $(X,g)$ be an Hadamard manifold, namely, a simply connected,
oriented, complete Riemannian manifold with a non-positively curved metric $g$.  Then, the exponential map $\exp_p :
T_pX \rightarrow X; v \mapsto \exp_p v$ is a diffeomorphism.

Let $B(q;r)$ be a geodesic ball in $X$ of center $q$ and radius $r$.
The volume of $B(q; r)$, ${\rm Vol}\hspace{0.5mm}B(q;r)$, is given by integral over $B(q;r)$ of  the Riemannian volume element
$dv_g = \sqrt{\det(g_{ij})(p)}\,dx^1\wedge\cdots\wedge dx^n$ and is expressed as
\begin{equation}
 {\rm Vol}\ B(q;r)  = \int_0^r dt \int_{S(q;t)} dv_{S(q;t)} = \int_0^r  {{ \int_{u\in S_qX} t^{n-1} J(t,u)\
 dudt}},
\end{equation}
with respect to $S_qX$, the space of unit tangent vectors in $T_qX$.  Here $J(t,u)$ is a function given by
\begin{equation}
  J(t,u) = t^{-(n-1)} \sqrt{\det\left(\langle Y_i(t),Y_j(t)\rangle\right)},
 \end{equation}
defined in the following way. Let $\gamma(t) = \exp_q t u$ be
a unit speed geodesic and $\{ e_1(t), \cdots,
e_n(t)\}$ be a parallel orthonormal frame field along $\gamma$ such
that 

\noindent
$\{e_1(0),\cdots, e_n(0)\}$ is an orthonormal frame at
$\gamma(0)$, $e_1(0) = u$ and let $\{Y_2(t),\cdots, Y_n(t)\}$ be a set of perpendicular Jacobi
vector fields along $\gamma$ satisfying $Y_i(0) = 0$ and
$Y'_i(0) = e_i(0)$. Each $Y_i$ is written as $Y_i(t) =  t (d \exp_q)_{t u} e_i(0)$ (see \cite[p.114]{Do}).
  We write $\Theta(t,u) = t^{n-1} J(t,u)$. Then $\displaystyle{\Theta(t,u) = \sqrt{\det\left(\langle Y_i(t),Y_j(t)\rangle\right)}}$ represents the volume density of $ (\exp_q)^{\ast} dv_{S(q; t)} $ over a geodesic sphere $S(q; t)$ with respect to local coordinates  $t, u$, induced by the map $\exp_q$. 
  
  Define an endomorphism, called a Jacobi tensor field 
  \begin{equation}\label{endomorphism}
  A(t) : \gamma^{\perp}(t) \rightarrow \gamma^{\perp}(t); A e_i(t) := Y_i(t), i = 2,\cdots,n.
  \end{equation} Then,  for any $t > 0$ $A(t)$ is invertible and self-adjoint with respect to the inner product $\langle\cdot, \cdot\rangle$ at $\gamma(t)$. We have $\Theta(t,u) = \det A(t)$. Moreover, an endomorphism ${\mathcal S}(t)$ defined by ${\mathcal S}(t) = A'(t) A^{-1}(t)$ and its trace ${\rm tr}\hspace{0.5mm}{\mathcal S(t)} = {\rm tr}\hspace{0.5mm} A'(t) A^{-1}(t)$, respectively, give  shape operator and mean curvature $\sigma(t) = \sigma(t,u)$ of a geodesic sphere $S(q;t)$ at $\gamma(t)$. 
 Notice that $\sigma(t,u) = \nabla_{\partial t}\,\Theta(t,u)/\Theta(t,u)$.
The Busemann function $b_{\gamma}$ associated to a geodesic $\gamma$ is defined as $\displaystyle{b_{\gamma}(x) = }$\, $\displaystyle{\lim_{t\rightarrow \infty} (d(x,}$\, $\displaystyle{\gamma(t)) - t)}$, $x\in X$.  The function $b_{\gamma}$ is $C^2$, convex and of unit gradient $\nabla b_{\gamma}$. Note that the Busemann function $b_{\gamma}$ fulfills $b_{\gamma}(\gamma(t)) = -t$ and $\vert b_{\gamma}(x)\vert \leq d(\gamma(0),x)$, $\forall x\in X$. The Hessian $\nabla d b_{\gamma}$ is positive semi-definite. Denote by $\partial X$ the ideal boundary of an Hadamard manifold $X$, which is the quotient space of all geodesic rays on $X$ modulo asymptotic equivalence among all geodesic rays.
Let $o\in X$ be a fixed point and $\theta\in \partial X$ be an arbitrary ideal boundary point.
Then, there exists a unique geodesic $\gamma : (-\infty,\infty) \rightarrow X$ such that  $\gamma\vert_{[0,\infty)}$ represents $\theta$ and $\gamma(0) = o$.  So one defines the Busemann function $b_{\theta} :\, X \rightarrow {\Bbb R}$ associated with $\theta$ as $b_{\theta}(x) := b_{\gamma}(x)$, $x\in X$. 

A level hypersurface of the Busemann function $b_{\theta}$ associated with $\theta$ is called a horosphere centered at $\theta$, denoted by ${\mathcal H}_{(\theta, t)} := b_{\theta}^{-1}(t)$. Then, 
since $b_{\theta}$ is surjective as a map to ${\Bbb R}$ with $\vert \nabla b_{\theta}\vert \equiv 1$, for any fixed $\theta\in\partial X$\, the function $b_{\theta}$ is a Riemannian submersion from $(X,g)$ onto $({\Bbb R}, dt^2)$ with fibre ${\mathcal H}_{(\theta,t)}$, $t\in{\Bbb R}$ such that at any $x\in{\mathcal H}_{(\theta,t)}$ the tangent space $T_xX$ admits an orthogonal direct sum; $T_xX = T_x{\mathcal H}_{(\theta,t)}\oplus {\Bbb R}\hspace{0.3mm}\nabla  b_{\theta}\hspace{0.1mm}_x$. Therefore, the metric $g$ of $X$ is represented by $g_x= g_t \oplus dt^2$, $g_t = g\vert_{{\mathcal H}_{(\theta,t)}}$ and hence the volume form $dv_g$ by $dv_g = dt \cdot d\sigma_t$, $d\sigma_t$ is the Riemannian volume form of the metric $g_t$ of ${\mathcal H}_{(\theta,t)}$.\, Then, integration of a function $f$ over $X$ is represented  by integrating along the fibres, i.e., horospheres and then over the real line as
\begin{eqnarray}\label{horosphericalintegration}
\int_X f(x) dv_g(x) = \int_{-\infty}^{+\infty} dt \int_{x\in{\mathcal H}_{(\theta,t)}} f(x) d\sigma_t.
\end{eqnarray}

The minus signed Hessian of $b_{\theta}$, $-\hspace{0.5mm}\nabla d b_{\theta}$ and its trace $\Delta b_{\theta} = - {\rm tr}\hspace{0.5mm}\nabla d b_{\theta}$ give the second fundamental form and the mean curvature $\tau$ of a horosphere associated with the Busemann function, respectively.   Let $\{ S(o;t); t > 0\}$ be a family of geodesic spheres and $\{ {\mathcal H}_{(-\theta,t)}\, \vert\, t > 0\}$ a family of horospheres associated with the Busemann function $b_{-\theta}$  passing $\gamma(t)$, where $-\theta\in \partial X$ denotes the ideal point represented by the reversed geodesic $\gamma^-$ of $\gamma$. Then $S(o; t)$ osculates ${\mathcal H}_{(-\theta,t)}$ at $\gamma(t)$. Under this situation we have the inequality (\cite{ISS})
 along $\gamma(t)$  
  \begin{equation*}
  \vert \sigma(t,u) - \tau(t)\vert < \frac{n-1}{t},\ \forall t > 0.
  \end{equation*}

  When each horosphere has constant mean curvature  and this constant is common (denote this value by  $-\hspace{0.5mm}Q$) over all horospheres (such an Hadamard manifold is called asymptotically harmonic, due to Ledrappier \cite{L}), volume entropy $\rho(X,g)$, defined by $\displaystyle{\rho(X,g) = }$\, $\displaystyle{\lim_{r\rightarrow\infty} \frac{1}{r} \log {\rm Vol} B(q;r)
  }$ is equal to $Q \geq 0$ from \cite{ISS}. 
  
\section{A harmonic Hadamard manifold and spherical functions}

\begin{definition} A Riemannian manifold $(X, g)$ is called harmonic, if the volume
density $\Theta(r,u) := r^{n-1} J(r,u)$ is a function of the distance function  $r$, independent of $u\in S_qX$ for any $q\in X$.
\end{definition}

\begin{theorem}[{cf. \cite[6.21 Proposition]{Besse}, \cite[Lemma 1.1]{Sz}}]\label{equiv_harmonicmfd}\rm
A Riemannian manifold is harmonic if and only if one of the following conditions holds
\begin{enumerate}
\item mean curvature $\sigma(r,u)$ of geodesic sphere $S(q;r)$ is a radial function.
\item there exists a radial function $f = f(r)$ for which
the function $h(p) = f(r(p))$ on $X$ satisfies the Laplace equation $\Delta h = 0$ and
\item the averaging operator ${\mathcal{MV}}_q$ commutes with $\Delta$.
Here, for a smooth function $f$, ${\mathcal{MV}}_q(f)$ is a radial function over $X$ whose value is the average of $f$ on $S(q;r)$;
\begin{equation}\label{avrgop}
{\mathcal{MV}}_q(f)(r) := \frac{1}{\int_{S(q;r)} dv_{S(q;r)}}\, \int_{p\in S(q;r)} f(p) \  d v_{S(q;r)}.
\end{equation}

\end{enumerate}
\end{theorem}

Let $(X,g)$ be a Hadamard manifold. Assume that $(X,g)$ is harmonic and of positive volume entropy $\rho(X,g) = Q > 0$.

\begin{note}{
Notice that a harmonic Hadamard manifold having volume entropy $\rho(X,g) = 0$ is flat, due to \cite[Thm 4.2]{RS}.
}
\end{note}

We will define spherical functions on  a harmonic Hadamard manifold $(X,g)$ as follows.  Let $\Delta = - g^{ij} \nabla_i \nabla_j$ denote the Laplace-Beltrami operator of $(X,g)$.

\begin{definition}
{
A spherical function $\varphi$ is a radial eigenfunction of the 
{Laplace-Beltrami operator} $\Delta$, satisfying $\varphi(o) = 1$
at a reference point $o\in X$.
}
\end{definition}

Since $(X,g)$ is harmonic, $(X,g)$ must be asymptotically harmonic. Then, each Busemann function $b_{\theta}$, which is geodesically defined over $X$, normalized at the reference point $o$ and parametrized with respect to $\theta\in \partial X$, an ideal point at infinity, satisfies $\Delta\hspace{0.5mm}b_{\theta}(\cdot) \equiv -\hspace{0.5mm}Q$, since  the minus signed Hessian of $b_{\theta}$ gives the second fundamental form of a horosphere associated with $b_{\theta}$. Using $b_{\theta}$ we define for a fixed $\theta\in\partial X$ a function on $X$;
\begin{equation}
  P(x,\theta) := \exp\{ -\hspace{0.5mm}Q\hspace{0.5mm}b_{\theta}(x)\},\, x\in X, \, \theta\in\partial X.
\end{equation}We have then

\begin{lemma}$P(x,\theta)$ is a positive valued, harmonic function of
$x\in X$ for any $\theta\in\partial X$.
\end{lemma}

Let $\lambda$ be a complex number. We
define $P_{\lambda}(x,\theta)$ a complex valued function on $X$ as
\begin{equation}
P_{\lambda}(x,\theta) := P(x,\theta)^{\left(\frac{1}{2} - i
\frac{\lambda}{Q}\right)} = \exp\left\{- \left(\frac{Q}{2} - i
\lambda\right) \hspace{0.5mm}b_{\theta}(x) \right\} .
\end{equation}Then, 

\begin{lemma} $\displaystyle \Delta P_{\lambda}(\cdot,\theta) = \left(\lambda^2 + \frac{Q^2}{4}\right)\, P_{\lambda}(\cdot,\theta)$ for any $\theta\in\partial X$.
\end{lemma}
Namely, $P_{\lambda}(x,\theta)$ is an eigenfunction of $\Delta$ of eigenvalue $\lambda^2 + \frac{Q^2}{4}$.
\begin{proof}
In fact, $\nabla_kP_{\lambda} = -(Q/2-i\lambda) \nabla_k b_{\theta}\cdot\hspace{0.5mm} P_{\lambda}$
 and
 
 \begin{eqnarray*}
 \nabla_j\nabla_k P_{\lambda} &=& \left\{-\left(\frac{Q}{2}-i\lambda\right)\right\}^2 \nabla_j b_{\theta}\, \nabla_k b_{\theta}\cdot\hspace{0.5mm} P_{\lambda} 
 - \left(\frac{Q}{2}-i\lambda\right) \nabla_j\nabla_k b_{\theta}\cdot\,P_{\lambda}
 \end{eqnarray*}
 so
 \begin{eqnarray*}\Delta P_{\lambda} = - \sum g^{jk} \nabla_j\nabla_k P_{\lambda} 
= \left\{-\left(\frac{Q}{2}- i\lambda\right) \Delta b_{\theta} -\left(\frac{Q}{2}-i\lambda\right)^2\vert \nabla b_{\theta}\vert^2\right\} P_{\lambda}.
\end{eqnarray*}
Since $\vert\nabla b_{\theta}\vert^2= 1$ and $\Delta b_{\theta} = -Q$, one has $\Delta P_{\lambda} =  \left(\frac{Q^2}{4}+\lambda^2\right) P_{\lambda}$.
\end{proof}

We obtain, therefore, the spherical function
$\varphi_{\lambda} = \varphi_{\lambda}(r)$ with eigenvalue $\nu =
\lambda^2 + \frac{Q^2}{4}$, by taking spherical average over the geodesic sphere $S(o;r)$ of $P_{\lambda}(x,\theta)$ for $\lambda\in \mathbb{C}$, since $\Delta$ commutes with the operator ${\mathcal{MV}}_o$ over $X$ (see Theorem \ref{equiv_harmonicmfd} (iii);
\begin{eqnarray}
\varphi_{\lambda}(r) := {\mathcal{MV}}_o(P_{\lambda}(\cdot,\theta))(r), \hspace{2mm}r> 0.
\end{eqnarray}
So, $\varphi_{\lambda}(r)$ satisfies
 \begin{equation}\label{sphricalequa}
 \Delta \varphi_{\lambda} = - \left({\frac{d^2 \varphi_{\lambda}}{d r^2}} + \sigma(r) \frac{d\varphi_{\lambda}}{d r} \right) = \nu\hspace{0.5mm}\varphi_{\lambda},\hspace{2mm}\varphi_{\lambda}(0) = 1, \, \varphi_{\lambda}'(0) = {
 {0}}.
\end{equation}{Here, $\displaystyle   - \left(\frac{d^2 }{d^2 r} + \sigma(r) \frac{d }{d r} \right) = \Delta^{rad}$, the radial part of $\Delta$ in terms of polar coordinates(\cite[(1.3)]{Sz}) and $\sigma(r)$ denotes the mean curvature of a geodesic sphere $S(o;r)$. 

\begin{note}We extend each spherical function $\varphi_{\lambda}$ as  an even function on ${\Bbb R}$.\, 
From (\ref{sphricalequa}) $\varphi_{\lambda} = \varphi_{\mu}$ if and only if $\lambda = \pm\hspace{0.5mm}\mu$. So, $\varphi_{\lambda}(r) = \varphi_{-\lambda}(r)$ and $\overline{\varphi_{\lambda}(r)} = \varphi_{\overline{\lambda}}(r)$,\, $r \in {\Bbb R}$. $\varphi_{\lambda}(r)$ is real valued,  when $\lambda\in {\Bbb R}$.
Further, for $\displaystyle{\lambda={\pm i\hspace{0.5mm}\frac{Q}{2}}}$\, one has $\varphi_{\lambda}(r)\equiv 1$.
\end{note}
The boundedness of the spherical functions $\varphi_{\lambda}(r)$ is given as
\begin{lemma}\label{estimationofspherical}\rm If
$\vert\Im \lambda\vert \leq \frac{Q}{2}$, then $\vert \varphi_{\lambda}(r)\vert \leq 1$ for any $r \geq 0$.
\end{lemma}
\begin{proof}\rm
We give a proof by following the argument of \cite[p.81]{R}.  Assume $\lambda = \xi+i \eta$\hspace{2mm} with $\vert \eta\vert < \frac{Q}{2}$. Then $\frac{Q}{2} \pm \eta > 0$. Further $-\left(\frac{Q}{2}-i\lambda\right)= -\left(\frac{Q}{2}+\eta\right)+i\xi$ so that $\left\vert\exp\left\{i\xi b_{\theta}(x)\right\}\right\vert = 1$ and hence
\begin{eqnarray}\label{sphericalinequality}
\vert \varphi_{\lambda}(r)\vert &\leq& \frac{1}{{\rm Vol} (S(o;r))}\, \int_{x\in S(o;r)}  \exp\left\{-\left(\frac{Q}{2}+ \eta\right) b_{\theta}(x)\right\} dv_{S(o;r)}(x)\\ \nonumber
&=& \frac{1}{{\rm Vol}(S(o;r))}\, \int_{x\in S(o;r)}  \exp\{-Q\,b_{\theta}(x)\}^{\frac{\frac{Q}{2}+ \eta}{Q}} dv_{S(o;r)}(x).
\end{eqnarray}Apply the H\"older inequality with respect to the conjugate indices $k = \frac{Q}{Q/2+ \eta}$, $\ell = \frac{Q}{Q/2- \eta}$; 
\begin{eqnarray*}\frac{1}{{\rm Vol}(I)}\,\int_I f(x)^{1/k} dx \leq \left\{\frac{1}{{\rm Vol}(I)}\, \int_I f(x) dx\right\}^{1/k}
\end{eqnarray*} to the last integral term of (\ref{sphericalinequality}), we obtain
\begin{eqnarray}
\vert\varphi_{\lambda}(r)\vert \leq \{{\mathcal {MV}}_o(\exp(-Q b_{\theta}(\cdot)))(r)\}^{1/k} = \{\varphi_{i Q/2}(r)\}^{\frac{\frac{Q}{2}+ \eta}{Q}}\equiv 1
\end{eqnarray}so that $\vert \varphi_{\lambda}(r)\vert \leq 1$ for any $r \geq 0$, when $\vert \Im \lambda\vert < \frac{Q}{2}$. 
It is easily shown that $\vert \varphi_{\lambda}\vert (r) \leq 1$ for any $r \geq 0$, when $\vert \Im \lambda\vert = \frac{Q}{2}$. 
\end{proof}

\section{Spherical Fourier transform}
Let $(X,g)$ be a harmonic Hadamard manifold of volume entropy $Q > 0$.
  By using the spherical functions $\{\varphi_{\lambda}(r)\,\vert\, \lambda\in {\Bbb C}\}$ over $(X,g)$ we define the spherical Fourier transform.
  \begin{definition}\rm
 Let $f = f(r)$ be a radial smooth function on $X$ with compact support.
  \begin{eqnarray}
   {\mathcal H}(f)(\lambda) &:=& \omega_{n-1}\, \int_0^{\infty} f(r) \varphi_{\lambda}(r) \Theta(r) dr \\ \nonumber
   &=&  \int_X f(r(x)) P_{\lambda}(x,\theta) dv_g,\hspace{2mm}r(x) = d(o,x)
   \end{eqnarray}is called the spherical Fourier transform of $f$. The function ${\mathcal H}(f)(\lambda)$ thus defined, denoted by ${\hat f}(\lambda)$ for brevity, is an entire function of $\lambda\in{\Bbb C}$.
   \end{definition}
   Then, like the classical Fourier transform, the map ${\mathcal H}$ is linear and satisfies ${\mathcal H}(f\ast f_1)(\lambda) = {\mathcal H}(f)(\lambda)\ {\mathcal H}(f_1)(\lambda)$ for the convolution and ${\mathcal H}(\Delta f)(\lambda) = (\frac{Q^2}{4}+ \lambda^2) {\hat f}(\lambda)$ for any $f, f_1$, radial smooth functions of compact support.  The convolution of radial functions $f$, $f_1$ is defined by
   \begin{eqnarray*}
   (f\ast f_1)(d(x,y)) = \int_{z\in X} f(d(x,z)) f_1(d(z,y)) dv_g(z).
   \end{eqnarray*}  Refer to \cite[2,(2.8)]{Sz}. It is seen that the function $f\ast f_1$ is radial and of compact support (\cite[Prop. 2.1]{Sz}).

     We set the
 range and the domain of ${\mathcal H}$, respectively as the space 
 ${\mathcal C}^{\infty}_c(X)^{rad}$ of smooth radial functions $f = f(r)$ with compact support on $X$ and 
  the space ${\mathcal{PW}}{\Bbb C}_{even}$ of even entire functions $h = h(\lambda)$ of $\lambda\in {\Bbb C}$ of certain exponential type \cite{ADY, R, DiBlas}.
    We define precisely ${\mathcal{PW}}{\Bbb C}_{even} = \cup_{R>0} {\mathcal{PW}}{\Bbb C}^R_{even}$ where ${\mathcal{PW}}{\Bbb C}^R_{even}$ is the space of even, entire functions $h=h(\lambda)$ over ${\Bbb C}$ satisfying the following; for any $N\in {\Bbb N}$ there exists a constant $C_N > 0$ such that 
  \begin{eqnarray}
   \vert h(\lambda)\vert \leq C_N (1+\vert\lambda\vert)^{-N} \exp(R\vert \Im \lambda\vert),\hspace{2mm}\forall \lambda\in {\Bbb C}.
  \end{eqnarray}
  
   \begin{proposition}
   For any $f \in {\mathcal C}^{\infty}_c(X)^{rad}$ of ${\rm supp}(f) \subset B(o,R)$,\,  ${\mathcal H}(f)$ belongs to  
   $\mathcal{PW}{\Bbb C}^{R'}_{even}$ for $R' > R$.
   \end{proposition}
   
For this, refer to \cite[p.41]{DiBlas} and \cite{Ri}.

\begin{proof}
   Let $f= f(r)$ be a radial function of ${\rm supp}(f) \subset B(o,R)$. We define a function $g$ of $t\in{\Bbb R}$  by integrating $f$ along a horosphere ${\mathcal H}_{(\theta,t)}$ with respect to a fixed $\theta\in\partial X$ as
   \begin{eqnarray}\label{definitionofg}
   g(t) := \int_{x\in{\mathcal H}_{(\theta,t)}} \exp\left\{-\frac{Q}{2}t\right\} f(d(o,x)) d\sigma_t,\, \forall t.
   \end{eqnarray}Here ${\mathcal H}_{(\theta,t)}$ is the level hypersurface of $b_{\theta}$ of level $t$, that is, ${\mathcal H}_{(\theta,t)}=\{ y\in X\, \vert\, b_{\theta}(y) = t\}$\, and $d\sigma_t$ denotes the volume form of ${\mathcal H}_{(\theta,t)}$. Then by using (\ref{horosphericalintegration}) we have the equality
    \begin{eqnarray}
   {\hat f}(\lambda) = \int_{-\infty}^{+\infty} g(t) \ \exp\{i\lambda t\} dt,\, \lambda\in {\Bbb C}.
   \end{eqnarray}
   In fact, we represent ${\hat f}(\lambda)$ as an integral over $X$;
   \begin{eqnarray*}
   {\hat f}(\lambda) &=& \int_X f(r(x)) P_{\lambda}(x, \theta) dv_g \\
   &=& \int_X f(r(x)) \exp\left\{-\left(\frac{Q}{2}-i\lambda\right) b_{\theta}(x)\right\} dv_g.
   \end{eqnarray*}By applying the horospherical fibre structure of $X$, given at section \ref{prelim}  the above integral is expressed as, since $b_{\theta} = t$ over ${\mathcal H}_{(\theta,t)}$
    \begin{eqnarray}\label{classicalF}
{\hat f}(\lambda)
  &=&  \int_{-\infty}^{+\infty} dt \int_{x\in{\mathcal{H}_{(\theta,t)}}} f(d(o,x)) \exp\left\{-\left(\frac{Q}{2}-i\lambda\right) t\right\} d\sigma_t \\
  &=& \int_{-\infty}^{+\infty} dt \exp\left(i\lambda t\right) g(t).\nonumber
    \end{eqnarray}
    
    \vspace{2mm}For $x\in \mathcal{H}_{(\theta,t)}$ we have $\vert t\vert = \vert b_{\theta}(x)\vert \leq d(o,x)$, therefore $f(x)=0$ if ${\rm supp} f \subset B(o,R)$ and $\vert t \vert \geq R$.
      Thus $g(t) = 0$ for $\vert t\vert \geq R$ and the classical Paley-Wiener theorem shows that ${\hat f}$, which is by (\ref{classicalF}) the classical Fourier transform of $g$ with ${\rm supp}\, g \subset [-R,R]$, belongs to  the space ${\mathcal{PW}}{\Bbb C}^{R}_{even}$.
  \end{proof}
 
   \begin{proposition}The spherical Fourier transform 
${\mathcal H}$ maps ${\mathcal{C}}^{\infty}_c(X)^{rad}$ into  ${\mathcal{PW}}{\Bbb C}_{even}$.
\end{proposition}Notice that ${\mathcal H}$ is injective. Refer to \cite[Theorem 3.12]{PS}.

\begin{remark}\rm
Let $C^{\infty}_0$ be the space of smooth functions on ${\Bbb R}$ with compact support. For $f\in C^{\infty}_0$  its classical Fourier transform is given by 
$$\displaystyle{{\widehat f}^{cl}(\lambda) =  \int_{-\infty}^{\infty} f(t)\, e^{-i\lambda t} dt}$$
with the converse Fourier transform $h = h(\lambda) \mapsto {\tilde h}= {\tilde h}(t)$, given by
$$\displaystyle{{\tilde h}(t) = \frac{1}{2\pi} \int_{-\infty}^{\infty} h(\lambda)\, e^{i\lambda t} d\lambda}.$$
Note there is another fashion for defining the transform by taking $\displaystyle{1/\sqrt{2\pi}}$ as the normalization.
  
  The image of the classical Fourier transform is, by applying the argument of Phragmen-Lindel\"of principle, the space of entire functions on ${\Bbb C}$ of certain exponential type. 
For an even function $f\in C^{\infty}_0$  the Fourier transform is written by  the Fourier cosine transform as ${\hat f}^{cl}(\lambda) = \int_0^{\infty} f(t) \cos \lambda t\, dt$, so ${\hat f}^{cl}(\lambda)$ is an even function of $\lambda$. 
\end{remark}

\section{A harmonic Hadamard manifold of hypergeometric type}\label{harmonichypergeometric}

\vspace{2mm}Now set 
\begin{equation}\label{variabletransf}
z = - \sinh^2 \frac{r}{2}
\end{equation} in (\ref{sphricalequa}). Then, under this variable transformation
\begin{lemma}\label{convert} {
The equation (\ref{sphricalequa}) is converted into the following with respect to $z$
\begin{equation}\label{convertedequat}
  z (1-z) \frac{d^2 f}{d z^2} + \left\{
   \left(\frac{1}{2} \sinh r\right)\  \sigma(r) 
   + \frac{1}{2} \cosh r \right\} \frac{d\hspace{0.2mm}f}{d z} 
   - \nu f {=
 {0.}}
\end{equation}
}
\end{lemma}

 \vspace{2mm} Lemma \ref{convert} is obtained by a slightly straightforward computation and we omit a proof for Lemma \ref{convert}.

\begin{definition}[refer to Definition \ref{def}]\label{defhypergeometric}
A harmonic Hadamard manifold $(X,g)$ is said to be of hypergeometric type, if the converted differential equation (\ref{convertedequat}) is exactly a Gauss hypergeometric differential equation;   
 \begin{equation}\label{hypergeometricequation}
  z (1-z) \frac{d^2 f}{d z^2} + \left(c -(a+b+1)z\right)
   \frac{d\hspace{0.2mm}f}{d z}
   - a b\,  f = 0,
\end{equation}where $a, b, c \in \mathbb{C}$ are constants, and moreover $c \not= 0, -1, -2, \cdots$.
\end{definition}  

\begin{remark}
The variable transformation (\ref{variabletransf}) is the unique transformation under which the equation of eigenfunction is converted into the hypergeometric differential equation.
\end{remark}  

Then,  we have Theorem \ref{maintheorem} for a harmonic Hadamard manifold of hypergeometric type, as we will prove.
  
\begin{remark}\label{volumeentropy}
The harmonicity is homothetic invariant. However, Theorem \ref{maintheorem} indicates that the hypergeometric type harmonicity is not homothetic invariant, because from (\ref{k}) volume entropy of a harmonic Hadamard manifold $(X,g)$ of hypergeometric type satisfies necessarily $ Q > (n-1)/3$.  Notice that $Q$ must satisfy $\displaystyle{ \frac{n-1}{2} \leq Q \leq n-1}$ from Bishop comparison theorem with respect to the volume of geodesic spheres(\cite{ItohSatohpre}).
\end{remark}

Corollary \ref{gromov} of section \ref{intro} is a direct consequence of Theorem \ref{maintheorem}.

 \begin{proof}[Proof of Theorem \ref{maintheorem}]
We will show (ii) and then (i).

(ii):\hspace{2mm} The coefficient $h(r)$ of $\displaystyle \frac{d\hspace{0.2mm}f}{d z}$ is written from (\ref{convertedequat}) as
\begin{equation*}
h(r) = \frac{1}{2}\ \sinh r (\sigma(r) + \coth r).
\end{equation*}It has also another representation given as 
\begin{equation*}
\begin{split}
h(r) =& c - (a + b + 1) z = c - (a + b + 1) \left(- \sinh^2 \frac{r}{2}\right) \\
=& \left(c - \frac{1}{2}(a + b + 1)\right) + \frac{1}{2}(a+b+1) \cosh r,
\end{split}
\end{equation*}
by assuming that $(X,g)$ is of hypergeometric type. Therefore, we obtain the equality  \begin{equation}\label{sigmar}
  \sigma(r) = \frac{2}{\sinh r} \left(c - \frac{1}{2}(a+b+1)\right) + (a+b) \frac{\cosh r}{\sinh r}.
  \end{equation}
 
\begin{lemma}\label{meancurvature}
$\sigma(r)$ has the following asymptotical formulae;
\begin{align}
\sigma(r)=& \frac{n-1}{r} + o(1), \hspace{4mm}r \rightarrow +0,
\label{tendszero}\\
\sigma(r)=& Q + O\left(1/r\right), \hspace{4mm}r \rightarrow +{
{\infty,}}
\label{tendsinfty}
\end{align}
(refer for (\ref{tendszero}) to \cite[Lemma 12.2]{GrVan} and for (\ref{tendsinfty}) to \cite[Lemma 4.2]{ISS}).
\end{lemma}

\begin{remark}\rm
 (\ref{tendsinfty}) of Lemma \ref{meancurvature} is derived by the aid of Rauch comparison theorem (cf. \cite[Ch.10]{Do}), by comparing $(X,g)$ with  the euclidean space with respect to the sectional curvature. 
 \end{remark}

Letting $r \rightarrow +\infty$ in (\ref{sigmar}) $a + b = Q$, while $a b = \nu = \frac{Q^2}{4} + \lambda^2$. So, $a, b = \frac{Q}{2} \pm i \lambda$.
       From (\ref{tendszero}) letting $r \rightarrow 0$ leads $r \sigma(r) \rightarrow n-1$ and, from (\ref{sigmar})
\begin{equation*}
\begin{split}
r \sigma(r)&= \frac{2r}{\sinh r} \left(c - \frac{1}{2}(a+b+1)\right) + (a+b)  r \frac{\cosh r}{\sinh r} \\
&\rightarrow 2\left(c - \frac{1}{2}(a+b+1)\right) + (a+b) = 2c - 1
\end{split}
\end{equation*}
and consequently $c = n/2$.

      \vspace{2mm}Thus we obtain

\begin{equation*}
\frac{1}{2}\ \sigma(r) \sinh r + \frac{1}{2} \cosh r = \frac{n}{2} + (Q+1) \sinh^2 \frac{r}{2}.
\end{equation*}
So                 
\begin{equation*}
\begin{split}
  \sigma(r) \sinh r =& n + 2(Q+1) \sinh^2 \frac{r}{2} - \cosh r \\
  =& n \left(\cosh^2 \frac{r}{2} - \sinh^2 \frac{r}{2}\right) + {{2(Q+1)}} \sinh^2 \frac{r}{2} -\left( \cosh^2 \frac{r}{2} + \sinh^2 \frac{r}{2}\right)
\end{split}
\end{equation*}
and hence
\begin{equation*}
\sigma(r)\times 2 \sinh \frac{r}{2} \cosh \frac{r}{2} = (n-1) \cosh^2 \frac{r}{2} + (-n + 2 Q + 1) \sinh^2 \frac{r}{2},
\end{equation*}
so we have                                                                                                                                                                        \begin{equation*}
\begin{split}
 \sigma(r)=& \frac{(n-1)\cosh^2 \frac{r}{2}}{2 \sinh \frac{r}{2} \cosh \frac{r}{2}} + \frac{(2Q - n + 1) \sinh^2 \frac{r}{2}}{2 \sinh\frac{r}{2} \cosh \frac{r}{2}} \\
=& \frac{n-1}{2} \coth \frac{r}{2} + \left(Q - \frac{n-1}{2} \right)\tanh\frac{r}{2}.
\end{split}
\end{equation*}
 
 From the equality $\frac{\Theta'(r)}{\Theta(r)} = \sigma(r)$ we obtain easily
\begin{equation*}
\Theta(r) = 
k_g\ \sinh^{(n-1)}\frac{r}{2}\  \cosh^{(2Q-(n-1))} \frac{r}{2}
\end{equation*}
for a constant $k_g > 0$. The constant $k_g$ is given exactly as $-\ \frac{2^n}{3 Q -(n-1)}\hspace{0.5mm}{\rm Ric}_g$ from Ledger's formula
\begin{equation}\label{ledgerformula}
\left.\left(\frac{\Theta(r)}{r^{n-1}} \right)''\right|_{r=0} = - \frac{1}{3} {\rm Ric}_g,
\end{equation}
which is valid for a harmonic manifold (refer to \cite[6.38]{Besse} for the Ledger's formula).

Compute the left hand side of (\ref{ledgerformula}) as
\begin{equation*}
\frac{\Theta(r)}{r^{n-1}} = \frac{k_g}{2^{n-1}} \left(1+\frac{n-1}{3!}\left(\frac{r}{2}\right)^2+O(r^4)\right) \left(1+\frac{\ell}{2!}\left(\frac{r}{2}\right)^2+O(r^4)\right)
\end{equation*}($\ell := 2Q-(n-1)$) and then
$$\displaystyle{\frac{\Theta(r)}{r^{n-1}} = \frac{k_g}{2^{n-1}} \left(1+ \left(Q -\frac{n-1}{3}\right) \frac{r^2}{4}+O(r^4) \right)
}$$
so that
$$
\left.\left(\frac{\Theta}{r^{n-1}}\right)''\right|_{r=0} = \frac{k_g}{2^n}\left(Q - \frac{n-1}{3}\right).
$$
We therefore obtain (\ref{k}).

                                                                                                                                                                                         \vspace{2mm}
 (i):\, The spherical function $f = \varphi_{\lambda}$ satisfies the hypergeometric differential equation. So, the function $f$ can be described as  (\ref{sphericalgausshyp}).
 \end{proof}

 \begin{remark}\rm  Although the formula (\ref{ledgerformula}) is shown from Ledger's formula, we will show it directly. It suffices to show
  \begin{eqnarray} \label{volumedensityformula} 
  \Theta(t) = t^{n-1}\left(1 - \frac{1}{3!} {\rm Ric}_g t^2 + O(t^3)\right).
  \end{eqnarray}For this, we write $\Theta(t)$ as $\Theta(t) = \det A(t)$ where $A(t)$ is the Jacobi tensor field along $\gamma$ defined at (\ref{endomorphism}). $A(t)$ satisfies $A(0)= O$, $A'(0)= {\rm id}_{u^{\perp}}$ ($u= \gamma'(0)$). 
   Then $A(t)$ is expanded with respect to $t$ as \begin{eqnarray*}(\tau^t_0)^{-1}\, A(t) \tau^t_0 = A(0) + A'(0) t + \frac{1}{2} A''(0) t^2 + \frac{1}{3!} A'''(0) t^3 +\cdots,
\end{eqnarray*} where $\tau_0^t$ is the parallel translation along $\gamma$ from $0$ to $t$. Since $A(t)$ satisfies the equation $A''(t) + R_{\gamma'(t)} A(t)= O$ with respect to the Jacobi operator $R_{\gamma'(t)}\, :\, \gamma^{\perp}(t) \rightarrow \gamma^{\perp}(t)$ which is associated to the Riemannian curvature tensor,\, the several coefficients of the expansion of $(\tau^t_0)^{-1}\, A(t) \tau^t_0$ other than $A(0)=O$, $A'(0)= {\rm id}_{u^{\perp}}$ are given by  $A''(0)= O$ and $A'''(0)= - R_{\gamma'(0)} A'(0) = -R_u$ and then we have
\begin{eqnarray*}\Theta(t) = t^{n-1} \, \det \left({\rm id}_{u^{\perp}} - \frac{1}{3!} R_u t^2 + O(t^3)\right)
 \end{eqnarray*} from which (\ref{volumedensityformula}) is derived, since $\det \left({\rm id}_{u^{\perp}} - \frac{1}{3!} R_u t^2 + O(t^3)\right) = 1 - \frac{1}{3!} {\rm tr}\, R_u t^2 + o(t^2)$ and ${\rm tr}\, R_u$ is the Ricci curvature ${\rm Ric}(u,u)$ of unit tangent vector $u$. Notice that from this argument  $(X,g)$ turns out to be  Einstein, since $\Theta(t)$ is independent of a choice of  $u$.
 \end{remark} 
 \begin{proof}[Proof of Theorem \ref{character}]
 This theorem is derived from (\ref{convertedequat}). In fact from Lemma \ref{meancurvature} one has $c_1 = (n-1)/2$ from an asymptotical property that \ $\lim_{r\rightarrow 0} r \sigma(r) = n-1$ and $c_1 + c_2 = Q$ from the fact $\lim_{r\rightarrow\infty} \sigma(r) = Q$. So, $c_2 = Q - (n-1)/2$. One substitutes this form of $\sigma(r)$ into (\ref{convertedequat}) from which one derives a Gauss hypergeometric equation.
 \end{proof}
 \begin{note} From the above argument one writes the radial part of the Laplace-Beltrami operator as
 \begin{equation}
 \Delta^{rad} = - \left[\frac{d^2}{d r^2} + \left\{\frac{ (n-1)}{2} \coth \frac{r}{2} + \left(Q - \frac{(n-1)}{2}\right) \tanh \frac{r}{2}\right\} \frac{d}{d r} \right].
 \end{equation}
Substituting  $t = r/2$,\,  $-4\hspace{0.5mm}\Delta^{rad}$ becomes the Jacobi operator 
 \begin{equation*}
   \frac{d^2}{d t^2} + \left\{(n-1) \coth t + (2Q - (n-1)) \tanh t\right\} \frac{d}{d t}. 
 \end{equation*}
So, spherical functions on $(X,g)$ are written by Jacobi functions, as shown in section \ref{jacobi} (see also \cite{ADY, Ko}).
\end{note}
                                                                                                                                                                                         \begin{proposition}\label{DR}  Let $(X,g)$ be a Damek-Ricci space. Then $(X,g)$ is a harmonic Hadamard manifold of hypergeometric type.
\end{proposition}For this proposition refer to {{\cite{ADY, R}}}.
\begin{remark}\label{damekriccikg} 
A Damek-Ricci space $S$ is a simply connected, solvable Lie
group with a left invariant Riemannian metric,  written as a semi-direct
product $S = A \ltimes N$ of $A \cong \mathbb{R}$ with a generalized
Heisenberg group $N$. The Lie algebra 
  $\mathfrak{n}$ of $N$ with an
inner product $\langle\cdot,\cdot\rangle$ is decomposed into
$\mathfrak{n} = \mathfrak{v} \oplus \mathfrak{z}$ with respect to
the center $\mathfrak{z}$ and its orthogonal complement
$\mathfrak{v} = \mathfrak{z}^{\perp}$. So, the Lie algebra $\mathfrak{s}$ of $S$ is $\mathfrak{s} = \mathfrak{v} \oplus
\mathfrak{z} \oplus \mathbb{R}$ and $\dim S = n$ is given as
$n = m + k + 1$, $m = \dim \mathfrak{v}$, $k = \dim \mathfrak{z}$. Via Cayley type transform, the volume density has the form of radial
function at the origin
\begin{equation}\displaystyle{\Theta(r) = 2^{m+k}\left(\sinh \frac{r}{2}\right)^{m+k} \left(\cosh \frac{r}{2}\right)^k }
\end{equation} (refer to \cite[(1.16)]{ADY}). 
$S$ is an Einstein
manifold. Ricci curvature of $S$ is $- (\frac{m}{4} + k)$ via the
formula (\ref{ledgerformula}) and volume entropy $Q = \frac{m}{2} + k$ so the constant $k_g = 2^{n-1} = 2^{m+k}$.
\end{remark}
                                                                                                                                                                                                                                                                                                                                                                                    \section{Hypergeometric functions}\label{hypergeometricfunctions}

We provide in this section several basic properties of Gauss hypergeometric functions which are adequate for the sequel.
 
 Let $F(a,b,c;z)$ be the Gauss hypergeometric function with parameters $a,b,c\in {\Bbb C}$($c\not= 0,-1,-2,\dots$) defined by the hypergeometric series
 \begin{eqnarray}
 \sum_{m=0}^{\infty} \frac{(a)_m(b)_m}{(c)_m}\cdot \frac{z^m}{m!},
  \end{eqnarray}
where
$$
\displaystyle{(a)_m := a(a+1)\dots(a+m-1) = \frac{\Gamma(a+m)}{\Gamma(a)}},\quad
(a)_0 := 1.
$$                                                                                                                                                                                                                                                                                                                                                                             
   This series is absolutely convergent for $\vert z\vert < 1$ and divergent when $\vert z\vert > 1$.
  Then, $F(a,b,c;z)= F(b,a,c;z)$ is analytic when $\vert z \vert < 1$. If $\Re(c-a-b) > 0$, the series is absolutely convergent when $\vert z\vert=1$ (cf. \cite[2.38]{WW}) and one has
  \begin{eqnarray}\label{convergent}
  F(a,b,c;1) = \frac{\Gamma(c)\Gamma(c-a-b)}{\Gamma(c-a)\Gamma(c-b)},
  \end{eqnarray}known as Gauss's Theorem(\cite[14.2]{WW}, \cite[15.1.20]{AbSt}). Here $\Gamma(z)$ is the Gamma function which is                                                                                                                                                                                                                                                                                                                                                                                                                                                                                                                                                                                                                                                                                                                                                                  analytic, except at the points $z= 0, -1,-2,\dots$, where $\Gamma(z)$ has poles(cf. \cite[12.10]{WW} for its definition and properties).  By analytic continuation $F(a,b,c;z)$ is considered as an analytic function for $z\in {\Bbb C}\setminus [1,\infty)$. The differentiation of $F(a,b,c;z)$ is given
   \begin{eqnarray}\label{differenhyper}
  \frac{d F}{d z}(a,b,c;z) = \frac{a b}{c}\, F(a+1,b+1,c+1;z),
  \end{eqnarray} refer to \cite[15.2.1]{AbSt} for this.

    $F(a,b,c;z)$ is a solution of the Gauss hypergeometric differential equation (\ref{hypergeometricequation}), regular at the singular point $z=0$. 
  
  There are many transformation formulae between hypergeometric functions among which we employ the following; 
  \begin{eqnarray}\label{connectionformula}
  F(a,b,c;z) &=& B_1 (-z)^{-a}\, F\left(a,1-c+a,1-b+a; \frac{1}{z}\right)\\
  & +& B_2 (-z)^{-b}\, F\left(b,1-c+b,1-a+b; \frac{1}{z}\right),\nonumber
  \end{eqnarray}
   \begin{equation}\label{connectionformula2} 
   B_1 = \frac{\Gamma(c)\Gamma(b-a)}{\Gamma(b)\Gamma(c-a)},\quad
   B_2 = \frac{\Gamma(c)\Gamma(a-b)}{\Gamma(a)\Gamma(c-b)}
  \end{equation}  
 (refer for this to \cite[15.3.7]{AbSt},\cite[2.10.(2),(5)]{Erdelyi}, \cite[15.51]{WW}) and  

\begin{equation*}
F(a,b,c;z) = (1-z)^{-a} \, F\left(a,c-b,c; \frac{z}{z-1}\right)
\end{equation*}
(refer to \cite[15.3.4]{AbSt},\cite[2.9.(3)]{Erdelyi}),  known as Kummer transformation formulae. The first one is significantly important for study of Jacobi transform and the spherical transform. The hypergeometric functions in the right hand of (\ref{connectionformula}) are solutions of the Gauss  hypergeometric differential equation, regular at $z=\infty$. The second formula appeared in \cite{R}.

\section{Jacobi Functions and Jacobi Transform}\label{jacobi}
                                                                                                                                                                                                                                                                                                                                                                                  \begin{definition}\rm                                                                                                                                                                                                                                                                                                                                                                                  The function $\phi^{(\alpha,\beta)}_{\mu} = \phi^{(\alpha,\beta)}_{\mu}(t) $ \, $(\alpha,\beta\in{\Bbb C}$, $-\alpha \not\in {\Bbb N}$) is called Jacobi function of order $(\alpha,\beta$), if it is an even smooth function on ${\Bbb R}$ which equals $1$ at $t=0$ and which satisfies the differential equation                                                                                                                                                                                                                                                                                                                     \begin{eqnarray}\label{jacobiequation}                                                                                                                                                                                                                                                                                                                                                                                 \left[                                                                                                                                                                                                                                                                                                                                                                                
\frac{d^2}{dt^2} + \left\{(2\alpha+1)\coth t + (2\beta+1)\tanh t\right\}\frac{d}{dt} + \left(\mu^2 + T^2\right                                                                                                                                                                                                                                                                                                                                                        ) \right] \phi_{\mu}(t) = 0.
\end{eqnarray}Here $T := \alpha+\beta+1$.\end{definition}

Let $\Omega_{(\alpha,\beta)}(t) := 2^{2T}\,\left(\sinh t\right)^{2\alpha+1}\, \left(\cosh t\right)^{2\beta+1}$ be the weight function associated with the Jacobi function $\phi^{(\alpha,\beta)}_{\mu}$ of order $(\alpha,\beta)$. 
For a smooth function $f$ on ${\Bbb R}$ with compact support the Jacobi transform ${\mathcal J}_{(\alpha,\beta)}$ of $f$ is defined by
                                                                                                                                                                                                                                                                                                                                                       
\begin{definition}[cf. \cite{ADY, Ko, Ko-x, Flensted-2}]\label{jacobitransform}                                                                                                                                                                                                                                                                                          \begin{eqnarray}                                                                                                                                                                                                                                                                                                                                            
\left({\mathcal J}_{(\alpha,\beta)}f \right)(\mu) := \int_0^{\infty} f(t) \phi_{\mu}^{(\alpha,\beta)}(t) \Omega_{(\alpha,\beta)}(t) dt,\hspace{2mm}\mu\in{\Bbb C}.
\end{eqnarray}                                                                                                                                                                                                                                                                                                                                                    
\end{definition}
                                                                                                                                                                                                                                                                                                                                                       The differential equation (\ref{jacobiequation}) is a second order equation  for which the point $t=0$ is a regular singularity. 
                                                                                                                                                                                                                                                                                                                                                        The Jacobi function $\phi^{(\alpha,\beta)}_{\mu}$  is the unique solution of (\ref{jacobiequation}), regular at $t=0$, and is expressed in terms of a hypergeometric function
\begin{eqnarray}                                                                                                                                                                                                                                                                                                                                                                                                                                                                                                                                                                                                                                                                                                                                                                     \phi^{(\alpha,\beta)}_{\mu}(t) = F\left(\frac{T-i\mu}{2}, \frac{T+i\mu}{2},\alpha+1; - \sinh^2 t   \right).
\end{eqnarray}
                                                                                                                                                                                                                                                                                                                                                                                    It is noted that $(\Gamma(\alpha+1))^{-1} \phi_{\mu}^{(\alpha,\beta)}(t)$ is an entire function of $\alpha$, $\beta$ and $\mu$ (also for $\alpha = - 1,-2,\cdots$). See \cite{Ko-x} for this.
                                                                                                                                                                                                                                                                                                                                                                                  \begin{lemma}\label{coincide} Changing the variable as $t = r/2$, the spherical function $\varphi_{\lambda}= \varphi_{\lambda}(r)$ becomes the Jacobi function $\phi^{(\alpha,\beta)}_{\mu}(t)$ of order $\displaystyle (\alpha, \beta) = \left(\frac{n}{2}-1,Q- \frac{n}{2}\right)$ and $\mu= 2\lambda$;                                                                                                                                                                                                                                                                                                                                                                                    \begin{eqnarray}                                                                                                                                                                                                                                                                                                                                                                                   \label{sphericaljacobi} \varphi_{\mu/2}(r) = \phi_{\mu}^{(\alpha,\beta)}(r/2)                                                                                                                                                                                                                                                                                                                                                                                   . \end{eqnarray}                                                                                                                                                                                                                                                                                                                                                                                                                                                                                                                                                                                                                                                                                                                                                                     Furthermore, $\Theta(r) = C_g\hspace{0.5mm}\Omega_{\alpha,\beta}(r/2)$ and $Q = T (=\alpha + \beta + 1)$.                                                                                                                                                                                                                                                                                                                                                                                   
Here                                                                                                                                                                                                                                                                                                                                                                                                                                                                                                                                                                                                                                                                                                                                                                                                                                                                                                                                                                                                                                                                                                                                             
\begin{eqnarray}\label{c} C_g = 2^{-2Q} k_g  = - \frac{2^{n-2Q}}{3Q-(n-1)} {\rm Ric}_g.
\end{eqnarray}
                                                                                                                                                                                                                                                                                                                                                                                   \end{lemma}
                                                                                                                                                                                                                                                                                                                                                                                  
Therefore, the spherical transform is expressed as
                                                                                                                                                                                                                                                                                                                                                                                   \begin{eqnarray}
\left({\mathcal H}f\right)(\lambda) = 2\hspace{0.5mm}C\left({\mathcal J}_{\alpha,\beta)}(f\circ i_2)\right)(2\lambda),
\end{eqnarray}where $i_2 : {\Bbb R}\rightarrow{\Bbb R}; t \mapsto 2 t$ is the multiplication(cf. \cite[2.14]{ADY}).

Due to \cite[2.9(13)]{Erdelyi}, 
                                                                                                                                                                                                                                                                                                                                                                                    for $\mu \not\in - i {\Bbb N}$ another solution of (\ref{jacobiequation}) is given by
\begin{multline}
\label{anothersol}
\Phi_{\mu}^{(\alpha,\beta)}(t) = (2\sinh t)^{i\mu-T}\\
\times F \left(\frac{1}{2}(-\alpha+\beta+1-i\mu),\frac{1}{2}(T-i\mu), 1-i\mu; -(\sinh t)^{-2} \right),                                                                                                                                                                                                                                                                                                                                                                                                                                                                                                                                                                                                                                                                                                                                                                 \end{multline}(cf. \cite{Ko-x, Flensted}).  The function $\Phi_{\mu}^{(\alpha,\beta)}$ satisfies $\Phi_{\mu}^{(\alpha,\beta)}(t) = e^{(i\mu-T)t}(1+o(1))$, $t\rightarrow \infty$.

For $\mu \not\in {\Bbb Z}$ by using \cite[2.10(2), 2.10(5)]{Erdelyi}, $\phi_{\mu}^{(\alpha,\beta)}(t)$ is a linear combination of $\Phi_{\mu}^{(\alpha,\beta)}$, $\Phi_{-\mu}^{(\alpha,\beta)}$ which are linearly independent, as
\begin{eqnarray}\label{firstsecond}
\pi^{1/2}(\Gamma(\alpha+1))^{-1} \phi_{\mu}^{(\alpha,\beta)}(t) = \frac{1}{2} c_{\alpha,\beta}(\mu)\hspace{0.5mm}\Phi_{\mu}^{(\alpha,\beta)}(t)+ \frac{1}{2} c_{\alpha,\beta}(-\mu)\hspace{0.5mm} \Phi_{-\mu}^{(\alpha,\beta)}(t), 
\end{eqnarray}
where
\begin{eqnarray}\label{harish}
c_{\alpha,\beta}(\mu) =  \frac{2^{T}\,\Gamma(\frac{1}{2} i\mu)\Gamma(\frac{1}{2}(1+i\mu))}{\Gamma(\frac{1}{2}(T + i\mu))\Gamma(\frac{1}{2}(\alpha-\beta+1+i\mu))}.\end{eqnarray}
Here, since $\phi_{\mu}(t) \equiv \phi_{-\mu}(t)$, the coefficient of $\Phi_{-\mu}$ is $c_{\alpha,\beta}(-\mu)$. The formula (\ref{firstsecond}) comes from (\ref{connectionformula}), (\ref{connectionformula2}). Note for the Jacobi function $\Phi_{-\mu}(t)$ refer to \cite[2.9(9)]{Erdelyi}.                                                                                                                                                                                                                                                                                                                                                                                  
\begin{remark}
Let $(\alpha,\beta) = (-1/2,-1/2)$. Then
\begin{equation*}
\phi_{\mu}^{(-1/2,-1/2)}(t) = \cos \mu t, \qquad \Phi_{\mu}^{(-1/2,-1/2)}(t) = e^{i\mu t},
\end{equation*}
\begin{equation*}
\Omega_{-1/2,-1/2}(t) = 1,\qquad c_{-1/2,-1/2}(\mu)= 1.
\end{equation*}
 \end{remark}

  The Jacobi transform ${\mathcal J}_{\alpha,\beta}$ for $\alpha=\beta= -1/2$ is the classical Fourier cosine transform; for $f \in C_0^{\infty}$
  \begin{eqnarray}({\mathcal J}_{-1/2,-1/2}(f))(\mu) = (2/\pi)^{1/2} \int_0^{\infty} f(t) \cos \mu t\, dt
  \end{eqnarray}and the inversion formula for ${\mathcal J}_{-1/2,-1/2}$ is
    \begin{eqnarray}f(t) &=&  (2/\pi)^{1/2} \int_0^{\infty} ({\mathcal J}_{-1/2,-1/2}(f))(\mu) \cos \mu t\, d\mu \\ \nonumber
    &=&  (2\pi)^{-1/2} \int_{-\infty}^{\infty} ({\mathcal J}_{-1/2,-1/2}(f))(\mu) \cos \mu t\, d\mu.
  \end{eqnarray} 
 
                                                                                                                                                                                                                                                                                                                                                                                     \section{Green's formula}

 In this section and the subsequent sections, we verify the inversion formula not directly. We show the indirect version of the inversion formula, given in Proposition 10.6, or more precisely at the equality (\ref{formulax-3}) for any $h\in{\mathcal{PW}}{\Bbb C}_{even}$, by employing the method of G\"otze(\cite{G}) with respect to Green's formula for the Laplace-Beltrami operator $\Delta$ and Riemann-Lebesgue's lemma. 
 As Lemma \ref{hormandertrick} indicates, the map defined by  the right hand of the formula (1.9) has its range in ${\mathcal{C}}^{\infty}_c(X)^{rad}$  so that we obtain the inversion formula from (\ref{formulax-3}) by applying the injectivity of the spherical Fourier transform, shown in \cite{PS}.

\vspace{2mm}
The following is known as Green's formula.
\begin{proposition}\label{green}
Let $(M,g)$ be a compact, oriented, Riemannian manifold with boundary $\partial M$. Let $f_1$ and $f_2$ be smooth functions on $M$. Then
\begin{equation}
\int_M\left(f_1 \Delta {\overline f}_2 - {\overline f}_2 \Delta f_1  \right) dv_M = - \int_{\partial M} \left( f_1 \frac{\partial {\overline f}_2}{\partial \nu} - {\overline f}_2 \frac{\partial f_1}{\partial \nu} \right) dv_{\partial M}.
\end{equation}
$\frac{\partial f}{\partial \nu}$ denotes the normal derivative of $f$, defined by $\frac{\partial f}{\partial \nu}(x) := \langle \nabla f, \nu\rangle$, $x \in \partial M$, where $\nu$ is the outer unit normal field to $\partial M$.
\end{proposition}

Let $(X,g)$ be a harmonic Hadamard manifold. Let $o$ be a fixed point of $X$. Let $\varphi = \varphi(r)$ and $\psi =\psi(r)$ be complex valued, radial functions which are eigenfunctions of $\Delta$;
\begin{equation}
 \Delta \varphi = \left(\frac{Q^2}{4} + \lambda^2\right) \varphi,\hspace{2mm}\Delta \psi = \left(\frac{Q^2}{4} + \mu^2\right) \psi,
\end{equation}
where $Q > 0$ is the volume entropy of $(X,g)$ and $\lambda$, $\mu \in {\Bbb C}$.

 Apply  Proposition \ref{green} to $\varphi = \varphi(r)$ and $\psi =\psi(r)$ over a geodesic ball $M = B(o;r) \subset X$. We have then, since $\nu = \nabla r$ over $S(o;r) = \partial B(o;r)$
 \begin{lemma}\label{greenformula1}
 \begin{eqnarray*}
(\lambda^2 - {\overline{\mu}}^2) \int_{B(o;r)} \varphi(r) {\overline \psi}(r) dv_{B(o;r)}
&=& \omega_{n-1}\,\left( \varphi(r) {\overline \psi}'(r) - \varphi'(r) {\overline \psi}(r) \right) \Theta(r)\\
&=& \omega_{n-1} {\mathcal W}_{\Theta}(\varphi(r),{\overline \psi}(r)) .
\end{eqnarray*}
\end{lemma} Here\, ${\mathcal W}_{\Theta}(\varphi, \psi)(r) := \big\{\varphi(r)\ \psi'(r)- \varphi'(r) \psi(r)\big\} \Theta(r)$ is called the Wronskian of functions $\varphi(r)$ and $\psi(r)$ with respect to $\Theta(r)$.
The integration of a radial function $f= f(r)$ over $B(o;r)$ and over $\partial B(o;r)$ are respectively given by
$$\displaystyle{\int_{B(o;r)} f dv_{B(o;r)} = \int_0^r dt f(t)\cdot {\rm area}(S(o;t))}$$
and
$$\displaystyle{\int_{\partial B(o;r)} f(r) dv_{\partial B(o;r)} = f(r)\cdot {\rm area}(S(o;r))}.$$
Here ${\rm area}(S(o;t)) = \omega_{n-1} \Theta(r)$ with respect to the volume density
$$\Theta(r) = k_g\hspace{0.5mm}\sinh^{n-1} r/2\hspace{0.5mm}\cosh^{(2Q-(n-1))} r/2$$
 of $S(o;r)$ and  the volume $\omega_{n-1}$ of the unit $(n-1)$-sphere.
 \begin{note}\label{areaunitsphere}
 $\omega_{n-1} = \displaystyle{
\frac{2\pi^{n/2}}{\Gamma(\frac{n}{2})} }$.
 \end{note}

\section{Asymptotic formula}

Let $\lambda \in \Bbb{R}$. 
We consider the spherical functions $\varphi_{\lambda}(r)$  of Theorem \ref{maintheorem} 
 \begin{eqnarray}
 \varphi_{\lambda}(r)= F\left(\frac{Q}{2}-i\lambda, \frac{Q}{2}+i\lambda, \frac{n}{2}; z\right),\hspace{2mm}z=- \sinh^2 \frac{r}{2}. 
 \end{eqnarray}
 
 To apply Green's formula we need to compute the following term for $\lambda, \mu \in \Bbb{R}$, $\lambda \not= \pm\mu$
\begin{equation}\label{green-2}
 {\mathcal W}_{\Theta}(\varphi_{\lambda}, {\overline \varphi}_{\mu})(r) =  \left\{\varphi_{\lambda}(r)\,\left( {\overline \varphi}_{\mu}\right)'(r) 
  - \left(\varphi_{\lambda} \right)'(r)\, {\overline \varphi}_{\mu}(r)\right\}\, \Theta(r).
\end{equation}

 For each $\lambda$ we represent from (\ref{sphericaljacobi}), for $t= r/2$, $\varphi_{\lambda}$ as $\varphi_{\lambda}(r) = \phi_{2\lambda}^{(\alpha,\beta)}(t)$ and\,  $\displaystyle{\varphi'_{\lambda}(r) \left(:= \frac{d \varphi_{\lambda}}{dr}(r)\right)= \frac{1}{2} \phi'_{2\lambda}(t)}$ where $\displaystyle{\phi'_{\mu}(t)  := \frac{d}{dt}\phi_{\mu}(t)}$. 
In what follows, we abbreviate $\phi_{\mu}^{(\alpha,\beta)}(t)$ and $c_{\alpha,\beta}(\mu)$ as $\phi_{\mu}(t)$, and $c(\mu)$, respectively.
From  (\ref{firstsecond}) (cf. \cite{Ko-x}) 
\begin{eqnarray}\label{linearcombination}\varphi_{\lambda}(r) = \frac{\Gamma(n/2)}{\pi^{1/2}} \cdot \frac{1}{2}\big\{c(2\lambda)\ \Phi_{2\lambda}(t) + c(-2\lambda)\ \Phi_{-2\lambda}(t)\, 
\big\}, \hspace{2mm}t= \frac{r}{2}, 
\end{eqnarray}where $\Phi_{\pm \mu}(t)$ and $c(\mu) = c_{\alpha,\beta}(\mu)$  are given at (\ref{anothersol}) and (\ref{harish}).  

\begin{note}\rm  ${\overline {c(\mu)}} = c(-\mu)$ for real $\mu$, and $\Gamma(\alpha+1) = \Gamma(n/2)$, since $\alpha = n/2-1$, $\beta= Q- n/2$ in the situation of spherical functions.
\end{note}

\begin{lemma}
\begin{eqnarray}\label{wronskian}{\mathcal W}_{\Theta}(\varphi_{\lambda}, {\overline \varphi}_{\mu})(r)  &=& \frac{\Gamma(n/2)^2}{ \pi}\hspace{0.5mm}C_g\hspace{0.5mm} \frac{1}{2\cdot 4}\, \left\{ c(2\lambda) {\overline{c(2\mu)}}\, {\mathcal W}_{\Omega}(\Phi_{2\lambda}, {\overline \Phi}_{2\mu})(t) \right.
 \\ \nonumber
&+& 
  c(2\lambda) {\overline{c(-2\mu)}}\, {\mathcal W}_{\Omega}(\Phi_{2\lambda}, {\overline \Phi}_{-2\mu})(t) 
\\ \nonumber
&+&   c(-2\lambda) {\overline{c(2\mu)}}\, {\mathcal W}_{\Omega}(\Phi_{-2\lambda}, {\overline \Phi}_{2\mu})(t)  \\ \nonumber
&+&   \left.c(-2\lambda) {\overline{c(-2\mu)}}\, {\mathcal W}_{\Omega}(\Phi_{-2\lambda}, {\overline \Phi}_{-2\mu})(t)\right\},\, t = \frac{r}{2},
\end{eqnarray} \end{lemma}
Here ${\mathcal W}_{\Omega}(\Phi_{2\lambda}, {\overline \Phi}_{2\mu})(t)$ is the Wronskian of $\Phi_{2\lambda}(t)$ and  ${\overline \Phi}_{2\mu}(t)$ associated with the weight function $\Omega(t)$ with respect to the variable $t > 0$. Notice the multiple factor  $\displaystyle{\frac{1}{2\cdot 4}\,   }$ of the right hand of (\ref{wronskian}) comes from the derivation coefficient as given by $\displaystyle{\varphi'_{\lambda}(r) = \frac{1}{2} \phi'_{2\lambda}(t)}$ together with twice of $1/2$ appeared in the form (\ref{firstsecond}).

\vspace{2mm}Asymptotic behavior of $\Phi_{\mu}(t)$ and $\Phi'_{\mu}(t)$ is obtained in \cite{Ko-x, Flensted-2} as
\begin{lemma}\label{estimationofphi} As $t \rightarrow \infty$,
\begin{eqnarray}\Phi_{\mu}(t) &=& e^{(i\mu-Q)t}\left(1+o(1)\right), \\
\Phi'_{\mu}(t) &=& \left((i\mu-Q)+ 2^3 e^{-2t} \right)\, e^{(i\mu-Q)t} \left(1+o(1)\right).
\end{eqnarray}
\end{lemma} In fact
\begin{multline}                                                                                                                                                                                                                                                                                                                                                             
\Phi_{\mu}^{(\alpha,\beta)}(t) = (e^t - e^{-t})^{i\mu-Q}\\                                                                                                                                                                                                                                                                                                                                                                            \times F\left(\frac{1}{2}(-\alpha+\beta+1-i\mu),\frac{1}{2}(Q-i\mu), 1-i\mu; -(\sinh t)^{-2} \right),
\end{multline} so the first formula is obvious. The second one is obtained as follows. 
\begin{eqnarray*}\Phi'_{\mu}(t) = \big\{(e^t - e^{-t})^{i\mu - Q}\big\}' F(z) + \left(e^t - e^{-t}\right)^{i\mu - Q}
\frac{dz}{dt}\, \frac{d F}{dz}(z),
\end{eqnarray*}where $\displaystyle z = - \frac{1}{\sinh^2 t}$ and $\displaystyle{\frac{dz}{dt} =  2(\sinh t)^{-3} \cosh t}$ so that by applying the differentiation formula (\ref{differenhyper}) one gets the second asymptotic formula. 

\vspace{2mm}

The Wronskian ${\mathcal W}_{\Omega}(\Phi_{2\lambda}, {\overline \Phi}_{2\mu})(t)$ is represented asymptotically, by the aid of  Lemma \ref{estimationofphi} as

\begin{lemma}
\begin{eqnarray}\label{wrons-1}{\mathcal W}_{\Omega}(\Phi_{2\lambda}, {\overline \Phi}_{2\mu})(t) &=& - i(2\lambda+ 2\mu)\, \exp\{i(2\lambda-2\mu)t\}(1+o(1)),
\end{eqnarray} \hspace{2mm}as $t \rightarrow \infty$.
\end{lemma}
\begin{proof}
\begin{eqnarray}
\Phi_{2\lambda}(t)\cdot {\overline{\Phi_{2\mu}'(t)}} &=& \exp\{(i2\lambda-Q)t\} {\overline{(i2\mu-Q)}} \exp\{{\overline{(i2\mu-Q)t}}\}(1+o(1))\\ \nonumber
&=& \left(-i 2\mu - Q\right) \exp\{\left(i(2\lambda-2\mu)-2Q\right)t\}(1+o(1))
\end{eqnarray}and similarly
\begin{eqnarray*}\Phi'_{2\lambda}(t)\cdot {\overline{\Phi_{2\mu}(t)}} &=& (i2\lambda-Q)\exp\{(i2\lambda-Q)t\}  \exp\{{\overline{(i2\mu-Q)t}}\}(1+o(1))\\
&=& \left(i 2\lambda - Q\right) \exp\{\left(i(2\lambda-2\mu)-2Q\right)t\}(1+o(1))
\end{eqnarray*}so that (\ref{wrons-1}) is obtained, since $\Omega(t) = e^{2Q t}(1+o(1))$.
\end{proof}
Other Wronskians are given similarly as
\begin{eqnarray*}
{\mathcal W}_{\Omega}(\Phi_{2\lambda}, {\overline \Phi}_{-2\mu})(t) 
&=& - 2i(\lambda-\mu)\, \exp\{\left(2i(\lambda+\mu)\right)t\}(1+o(1)),\\ \nonumber
{\mathcal W}_{\Omega}(\Phi_{-2\lambda}, {\overline \Phi}_{2\mu})(t) 
&=& - 2i(-\lambda+ \mu)\, \exp\{\left(2i(-\lambda-\mu)\right)t\}(1+o(1)),\\ \nonumber
{\mathcal W}_{\Omega}(\Phi_{-2\lambda}, {\overline \Phi}_{-2\mu})(t) &=& 2i(\lambda+ \mu)\, \exp\{\left(-2i(\lambda-\mu)\right)t\}(1+o(1)).
\end{eqnarray*}
 
Therefore, by noticing $\Theta(r) = C_g\hspace{0.5mm} \Omega(t)$, $t = r/2$ we have

\begin{lemma}The Wronskian term ${\mathcal W}_{\Theta}(\varphi_{\lambda},{\overline{\varphi_{\mu}}})(r)$ is given, as $r$ and hence $t\rightarrow +\infty$
\begin{eqnarray*}\hspace{12mm}{\mathcal W}_{\Theta}(\varphi_{\lambda},{\overline{\varphi_{\mu}}})(r) &=&\frac{\Gamma(n/2)^2}{ \pi}\hspace{0.5mm}C_g\hspace{0.5mm}\frac{1}{2\cdot 4}\, 
\left\{ c(2\lambda){\overline{c(2\mu)}}\, \left(-2i(\lambda+\mu)\right) e^{\left(2i(\lambda-\mu)\right)t}\right. \\ \nonumber
&+& c(2\lambda)c(2\mu)\, \left(-2i(\lambda-\mu)\right) e^{\left(2i(\lambda+\mu)\right)t} \\ \nonumber
&+& {\overline{c(2\lambda)}}{\overline{c(2\mu)}}\, \left(2i(\lambda-\mu)\right) e^{\left(-2i(\lambda+\mu)\right)t} \\ \nonumber
&+& \left.{\overline{c(2\lambda)}}c(2\mu)\, \left(2i(\lambda+\mu)\right) e^{\left(-2i(\lambda-\mu)\right)t}  
\right\}(1+o(1)).
\end{eqnarray*}
\end{lemma}

Therefore, since $t = r/2$,
\begin{eqnarray}& & \hspace{10mm}{\mathcal W}_{\Theta}(\varphi_{\lambda},{\overline{\varphi_{\mu}}})(r)\\ \nonumber
 &=& \frac{\Gamma(n/2)^2}{ \pi}\hspace{0.5mm}C_g\hspace{0.5mm}\frac{2}{2\cdot 4}\left\{ c(2\lambda){\overline{c(2\mu)}} \{-i(\lambda+\mu)\}\, \left(\cos (\lambda-\mu)r + i \sin(\lambda-\mu)r\right)\right. \\ \nonumber
&+&c(2\lambda) c(2\mu) \{-i(\lambda-\mu)\}\, \left(\cos (\lambda+\mu)r + i \sin(\lambda+\mu)r\right) \\ \nonumber
&+&{\overline{c(2\lambda)}}{\overline{c(2\mu)}} \{i(\lambda-\mu)\}\, \left(\cos (\lambda+\mu)r - i \sin(\lambda-\mu)r\right) \\ \nonumber
&+&\left.{\overline{c(2\lambda)}}c(2\mu) \{i(\lambda+\mu)\}\, \left(\cos (\lambda-\mu)r - i \sin(\lambda-\mu)r\right)\right\}(1+o(1))
\end{eqnarray}
As a consequence
\begin{eqnarray}\label{computewronskian} & &{\mathcal W}_{\Theta}(\varphi_{\lambda},{\overline{\varphi_{\mu}}})(r)\\ \nonumber
&=& \frac{\Gamma(n/2)^2}{ \pi}\hspace{0.5mm}C_g\hspace{0.5mm}\frac{1}{4}\Big[\{i(\lambda+\mu) \cos(\lambda-\mu)r \}\{- c(2\lambda){\overline{c(2\mu)}} + {\overline{c(2\lambda)}} c(2\mu)\}
\\ \nonumber
&+& \{(\lambda+\mu) \sin(\lambda-\mu)r \}\{ c(2\lambda){\overline{c(2\mu)}} + {\overline{c(2\lambda)}} c(2\mu)\}
\\ \nonumber
&+& \{i(\lambda-\mu) \cos(\lambda+\mu)r \}\{- c(2\lambda) c(2\mu) + {\overline{c(2\lambda)}} {\overline{c(2\mu)}}\}
\\ \nonumber
&+&\{(\lambda-\mu) \sin(\lambda+\mu)r \}\{ c(2\lambda) c(2\mu) + {\overline{c(2\lambda)}} {\overline{c(2\mu)}}\}
\Big](1+o(1))
\end{eqnarray}

Here the term $\{- c(2\lambda){\overline{c(2\mu)}} + {\overline{c(2\lambda)}} c(2\mu)\} $ tends to zero, when $\lambda \rightarrow \mu$ and $\{- c(2\lambda) c(2\mu) + {\overline{c(2\lambda)}} {\overline{c(2\mu)}}\}$ tends to zero, when $\lambda \rightarrow -\mu$ so that there exist smooth functions $L_1(2\lambda; 2\mu)$, $L_2(2\lambda;2\mu)$ of $\lambda$ such that
\begin{eqnarray}
\{- c(2\lambda){\overline{c(2\mu)}} + {\overline{c(2\lambda)}} c(2\mu)\} &=& (\lambda-\mu)\, L_1(2\lambda; 2\mu),\nonumber\\
\{- c(2\lambda) c(2\mu) + {\overline{c(2\lambda)}} {\overline{c(2\mu)}}\} &=& (\lambda+\mu)\, L_2(2\lambda;2\mu).
\end{eqnarray}
Here, for any fixed $\lambda$ $L_1$ and $L_2$ are given by
\begin{eqnarray}L_1(2\lambda;2\mu) &:=& \int_0^1\left\{- \frac{\partial c}{\partial \lambda} (2\lambda_1(s) )\hspace{0.5mm}ds\hspace{0.5mm} {\overline{c(2\mu)}} + \frac{\partial {\overline{c}}}{\partial \lambda} (2\lambda_1(s) )\hspace{0.5mm}ds\hspace{0.5mm}c(2\mu)\right\}, \\
L_2(2\lambda;2\mu) &:=& \int_0^1\left\{- \frac{\partial c}{\partial \lambda} (2\lambda_2(s) )\hspace{0.5mm}ds\hspace{0.5mm} c(2\mu) + \frac{\partial {\overline{c}}}{\partial \lambda} (2\lambda_2(s) )\hspace{0.5mm}ds\hspace{0.5mm}{\overline c}(2\mu)\right\}, 
\end{eqnarray} where $\lambda_1(s) = (1-s)\lambda + s \mu$ is a path joining $\lambda$ and $\mu$ and 
$\lambda_2(s) = (1-s)\lambda + s(- \mu)$ is a path joining $\lambda$ and $-\mu$.  This argument requires the following fact; 
Let $f= f(x)$ be a smooth function of $x\in {\Bbb R}$ satisfying $f(0) = 0$. Then there exists a smooth function $g = g(x)$ such that $f(x) = x \cdot g(x)$ and $g(0) = f'(0)$. In fact, let $g(x) = \int_0^1 \frac{df}{dx}(sx) ds$. Then $f(x) - f(0) = u(1)- u(0)= \int_0^1 \frac{d u}{ds}(s) ds$, where $u(s) :=  f(s x)$ for a given $x$.  By paying attention to the integral term $\int_0^1 \frac{d u}{ds}(s) ds$ more carefully, we may define $L_1$, $L_2$ more directly as
\begin{eqnarray}\label{ellonefunction}L_1(2\lambda;2\mu) &:=& \frac{-c(2\lambda) {\overline c}(2\mu)+ {\overline c}(2\lambda) c(2\mu)}{\lambda - \mu}, \hspace{2mm}\lambda\not=\mu, \\
&:=& - \frac{\partial c}{\partial \lambda} (2\mu) {\overline{c(2\mu)}} + \frac{\partial {\overline{c}}}{\partial \lambda} (2\mu)\hspace{0.5mm}c(2\mu),\hspace{2mm}\lambda=\mu,
\end{eqnarray} 
and similarly for $L_2(2\lambda;2\mu)$.  

Now we apply Lemma \ref{greenformula1} to the spherical functions $\varphi_{\lambda}(r)$, $\varphi_{\mu}(r)$ and then have the following by dividing (\ref{computewronskian}) by $\lambda^2 -\mu^2$ as
\begin{eqnarray}\label{ballintegral}& &\int_{x\in B(o;r)} \varphi_{\lambda}(r(x)) {\overline{\varphi_{\mu}}}(r(x)) dv_g(x) \\ \nonumber
&=&
\frac{1}{\lambda^2-\mu^2} \, \omega_{n-1}\, {\mathcal W}_{\Theta}(\varphi_{\lambda},{\overline{\varphi_{\mu}}})(r) \\ \nonumber
&=& \frac{\Gamma(n/2)^2}{ \pi}\hspace{0.5mm}C_g\hspace{0.5mm}\omega_{n-1}\, 
\frac{1}{4}\Big[i \cos(\lambda-\mu)r\cdot \hspace{0.5mm} L_1(2\lambda;2\mu)
+ i\cos(\lambda+\mu)r\cdot \hspace{0.5mm} L_2(2\lambda;2\mu) 
\\ \nonumber
&+& \frac{\sin(\lambda-\mu)r}{\lambda-\mu} \, \{ c(2\lambda){\overline{c(2\mu)}} + {\overline{c(2\lambda)}} c(2\mu)\}
\\ \nonumber
&+&\frac{\sin(\lambda+\mu)r}{\lambda+\mu} 
\{c(2\lambda) c(2\mu) + {\overline{c(2\lambda)}} {\overline{c(2\mu)}}\}
\Big](1+o(1)),
\end{eqnarray}as $r \rightarrow +\infty$.

Let $h = h(\lambda)\in {\mathcal{PW}}{\Bbb C}_{even}$ be an even entire function on ${\Bbb C}$ of exponential type.  We investigate the integral of $h= h(\lambda)$ with respect to the measure $\displaystyle{\frac{1}{\vert c(2\lambda)\vert^2} d\lambda }$ along the real line; 
\begin{eqnarray}\label{9.18}
\int_{-\infty}^{+\infty} h(\lambda)\frac{1}{\vert c(2\lambda)\vert^2} d\lambda\cdot \int_{x\in B(o;r)} \varphi_{\lambda}(r(x)) {\overline{\varphi_{\mu}}}(r(x)) dv_g,  
\end{eqnarray}for a fixed real $\mu$. From (\ref{ballintegral}) this is written as
\begin{eqnarray}\label{finalformula}& &\omega_{n-1}\,\frac{\Gamma(n/2)^2}{ \pi}\hspace{0.5mm}C_g\hspace{0.5mm}\frac{1}{4}\, \int_{-\infty}^{+\infty} d\lambda\, h(\lambda)\frac{1}{\vert c(2\lambda)\vert^2} \\ \nonumber
&\times&
\Big[i \cos(\lambda-\mu)r \hspace{0.5mm} L_1(2\lambda;2\mu)\\ \nonumber
&+& i\cos(\lambda+\mu)r \hspace{0.5mm} L_2(2\lambda;2\mu) 
\\ \nonumber
&+& \frac{\sin(\lambda-\mu)r}{\lambda-\mu} \, \{ c(2\lambda){\overline{c(2\mu)}} + {\overline{c(2\lambda)}} c(2\mu)\}
\\ \nonumber
&+&\frac{\sin(\lambda+\mu)r}{\lambda+\mu} 
\{c(2\lambda) c(2\mu) + {\overline{c(2\lambda)}} {\overline{c(2\mu)}}\}
\Big](1+o(1)){\color{red}{.}}
\end{eqnarray}

\section{Riemann-Lebesgue's lemma}
 \begin{theorem}[\cite{G}] \label{RLlemma}
 For any $h\in L^1([0,\infty))$ it holds
  \begin{equation}
 \lim_{t\rightarrow\infty}\int_0^{\infty} h(\lambda)\, \sin(\lambda t) \hspace{0.5mm}d\lambda = \lim_{t\rightarrow\infty}\int_0^{\infty} h(\lambda)\, \cos(\lambda t)\hspace{0.5mm}d\lambda =0
 \end{equation}
and \begin{equation}\label{secondriemann}
 \lim_{t\rightarrow\infty}\frac{1}{\pi}\ \int_0^{\infty} h(\lambda)\, \frac{\sin\left((\lambda-\mu) t\right)}{\lambda-\mu}\hspace{0.5mm}d\lambda = h(\mu), 
 \end{equation}for almost every $\mu > 0$.
 \end{theorem}
 
 We apply this theorem to functions defined on ${\Bbb R}$ by using Lebesgue's dominated convergence theorem as
 \begin{lemma}\label{riemannlebesgue}
 For any $h\in L^1({\Bbb R})$ it holds
 \begin{equation}
 \lim_{t\rightarrow\infty}\int_{-\infty}^{\infty} h(\lambda)\, \sin(\lambda t)\hspace{0.5mm} d\lambda = \lim_{t\rightarrow\infty}\int_{-\infty}^{\infty} h(\lambda)\, \cos(\lambda t)\hspace{0.5mm} d\lambda =0
 \end{equation} 
and 
 \begin{equation}\label{RL}
 \lim_{t\rightarrow\infty}\frac{1}{\pi}\ \int_{-\infty}^{\infty} h(\lambda)\, \frac{\sin((\lambda-\mu) t)}{\lambda-\mu}\hspace{0.5mm} d\lambda = h(\mu)
 \end{equation}for almost every fixed $\mu$.
 \end{lemma}
 
 \begin{remark}\label{bothellone}\rm (\ref{RL}) holds for any $\mu$, if $h=h(\lambda)$ and its classical Fourier transform belong to $L^1({\Bbb R})$. In fact, (\ref{RL}) follows from the inversion formula for the classical Fourier transform for $h$ by making use of
 \begin{eqnarray*}
 \frac{\sin(\lambda-\mu)t}{\lambda-\mu} = \frac{1}{2} \int_{-t}^t e^{-i(\lambda-\mu)x} dx.
  \end{eqnarray*} 
 
 \end{remark}
 
 \begin{proof}
 The integral $\int_{-\infty}^{\infty} h(\lambda) \sin(\lambda t)\hspace{0.5mm} d \lambda$ is written as
 \begin{equation*}\int_{-\infty}^{\infty} h(\lambda) \sin(\lambda t)\hspace{0.5mm} d \lambda = \int^{\infty}_0 \left(h(\lambda)- h(-\lambda)\right) \sin(\lambda t)\hspace{0.5mm} d \lambda.
 \end{equation*}
So, from Theorem \ref{RLlemma} this reduces to zero. The cosine formula 
 is similarly shown.
 
 To show (\ref{RL}) we let $\mu > 0$ without loss of generality.  Decompose (\ref{RL}) into
\begin{equation*}
\begin{split}
&\lim_{t\rightarrow\infty} \frac{1}{\pi} \int_{-\infty}^{\infty} h(\lambda)\, \frac{\sin((\lambda-\mu) t)}{\lambda-\mu}\hspace{0.5mm}d\lambda \\
=& \lim_{t\rightarrow\infty} \frac{1}{\pi} \int_{-\infty}^0 h(\lambda)\, \frac{\sin((\lambda-\mu) t)}{\lambda-\mu}\hspace{0.5mm}d\lambda + \lim_{t\rightarrow\infty} \frac{1}{\pi}\int_0^{\infty} h(\lambda)\, \frac{\sin((\lambda-\mu) t)}{\lambda-\mu}\hspace{0.5mm}d\lambda.
\end{split}
\end{equation*}
In the first integral term the function $\displaystyle{\frac{h(\lambda)}{\lambda-\mu}}$ belongs to $L^1$, since $\displaystyle{\vert h(\lambda)/(\lambda - \mu)\vert}$\, $\displaystyle{ \leq \vert h(\lambda)\vert/\mu
 }$. So, the first formula of Theorem \ref{RLlemma} is applied to see that the first integral term vanishes. Moreover, we can apply formula (\ref{secondriemann}) of Theorem \ref{RLlemma}  to the second integral term to obtain (\ref{RL}).
 \end{proof}

 \begin{lemma}\label{elonefinite}Let $h = h(\lambda)\in {\mathcal{PW}}{\Bbb C}_{even}$ be an even entire function of exponential type. The function of $\lambda\in {\Bbb R}$, defined by $\displaystyle{h(\lambda) \frac{L_i(2\lambda,2\mu)}{\vert c(2\lambda)\vert^2}}$,  belongs to $L^1(-\infty;\infty)$ for any fixed $\mu$, for $i=1,2$. 
  \end{lemma}

                                                                                                                                                                                                                                                                                                                                                                    \begin{proof}
 Since $h$ satisfies that for any $N\in {\Bbb N}$ there exist  $C_N > 0$ and $R>0$ such that, for $\lambda\in {\Bbb R}$,
 \begin{eqnarray}\vert h(\lambda)\vert \leq C_N \left(1+\vert\lambda\vert\right)^{-N} \exp\left( R\vert\Im \lambda\vert\right) = C_N \left(1+\vert\lambda\vert\right)^{-N}.
 \end{eqnarray} 
Here $\Im \lambda = 0$ for real $\lambda$.
 
 Now we provide the estimation for $c(\lambda)$ and $c(\lambda)^{-1}$ from \cite{Flensted,Ko-x}{\color{red}{.}}
 \begin{lemma}\label{estimationofc}There exists $K > 0$ such that
 \begin{eqnarray}\label{estimatec}
 \vert c(\lambda)\vert &\leq& K(1+\vert\lambda\vert)^{-(n-1)/2}, \\ 
\label{estimatec-2} \vert c(\lambda)^{-1}\vert &\leq& K(1+\vert\lambda\vert)^{(n-1)/2} 
  \end{eqnarray}for any real $\lambda$.
 \end{lemma}
 The estimation (\ref{estimatec}) is obtained by applying Corollary 9, \cite{Flensted} with respect to the estimation of $\lambda c(-\lambda)$;\, $\vert \lambda c(-\lambda)\vert \leq K (1+ \vert\lambda\vert)^{1-(p+q)/2}$, where the real numbers $p, q$ in the notation of \cite{Flensted} coincide with $(\alpha-\beta)/2 = 2(n-Q-1)$ and $2\beta+1 = 2Q-n+1$, respectively in our notation.  Although in \cite{Flensted} $p, q>0$ is assumed, we can get the above estimations by using directly Stirling formula for the Gamma function. 
 \begin{remark} (i)\hspace{2mm}$\lambda c(-\lambda)$ is a holomorphic in $\lambda$ and  (ii)\hspace{2mm}$c(-\lambda)^{-1}$ is continuous for real $\lambda$. (i) is obtained from the argument of the Wronskian(see \cite[Lemma 8]{Flensted}). In fact,  ${\mathcal W}_{\Omega}(\phi_{2\lambda},\Phi_{2\lambda})(t)$ is shown to be constant with respect to $t$ whose value is $\displaystyle{ \lim_{t\rightarrow \infty}{\mathcal W}_{\Omega}(\phi_{\lambda},\Phi_{\lambda})(t) }$ 
$\displaystyle{= \frac{\Gamma(n/2)}{\sqrt{\pi}} i\lambda c(-\lambda) }$ so that $\lambda c(-\lambda)$  is a holomorphic with respect to the parameter $\lambda\in{\Bbb C}$.   For (ii) Gamma function $\Gamma(z)$ and its reciprocal $1/\Gamma(z)$ are both meromorphic functions having simple poles at $z= - n$, $(n= 0,1,2,\cdots,)$ and simple zeros at $z= -n$,$(n= 0,1,2,\cdots)$, respectively. Refer for this to \cite[6.1.3]{AbSt}. 
 \end{remark}
 
 Therefore the $L^1$-estimation of $\displaystyle{h(\lambda)\ \frac{L_1(2\lambda;2\mu)}{c(2\lambda){\overline{c(2\lambda)}}} }$ is as follows. For a fixed $\mu$ and for $\lambda$ with $\lambda\not=\mu$ by using the notice for the description of $L_1$ at (\ref{ellonefunction})
\begin{eqnarray}
\left\vert h(\lambda)\ \frac{L_1(2\lambda;2\mu)}{c(2\lambda){\overline{c(2\lambda)}}}\right\vert &=& \vert h(\lambda)\vert \, \left\vert\frac{1}{c(2\lambda){\overline{c(2\lambda)}}}\right\vert\, \frac{\vert - c(2\lambda){\overline{c(2\mu)}}+{\overline{c(2\lambda)}} c(2\mu)\vert}{\vert \lambda-\mu\vert} \\ \nonumber
 &\leq& 2 \vert c(2\mu)\vert\, \frac{\vert h(\lambda)\vert}{\vert \lambda-\mu\vert} \, \frac{1}{\vert c(2\lambda)\vert} \\ \nonumber
 &\leq& 4 \vert c(2\mu)\vert K(1+\vert\lambda\vert)^{\{(n-1)/2 - 1\}} \vert h(\lambda)\vert, 
\end{eqnarray}
where the estimation of $c(-\lambda)^{-1}$ in Lemma \ref{estimationofc} is applied and $1/\vert\lambda-\mu\vert \leq  2/\vert\lambda\vert$ for any sufficiently large $\lambda$ is used.  Choose an integer $N> 0$ satisfying $N > (n-1)/2 -1+2$ so $(n-1)/2-1 - N < -2$ and $C_N > 0$. Then, that $h$ is exponential type implied 
 \begin{eqnarray}
\left\vert h(\lambda)\ \frac{L_1(2\lambda;2\mu)}{c(2\lambda){\overline{c(2\lambda)}}}\right\vert & \leq 4 \vert c(2\mu)\vert K\, C_N\, (1+\vert\lambda\vert)^{(n-1)/2 - 1-N}
 \end{eqnarray}from which the desired result is derived. A similar argument for $L_2$ completes the lemma. 
 \end{proof}

 By applying Riemann-Lebesgue's lemma and using Lemma \ref{elonefinite} we have
 \begin{eqnarray}\lim_{r\rightarrow\infty}\,\int_{-\infty}^{\infty} d \lambda \frac{h(\lambda)}{\vert c(2\lambda)\vert^2}\, L_1(2\lambda;2\mu) \, \cos(\lambda-\mu)r &=& 0,\\ \nonumber
 \lim_{r\rightarrow\infty}\,\int_{-\infty}^{\infty} d \lambda \frac{h(\lambda)}{\vert c(2\lambda)\vert^2}\, L_2(2\lambda;2\mu) \, \cos(\lambda+\mu)r &=& 0.
 \end{eqnarray}On the other hand we have
 \begin{eqnarray}\label{hmu}
 & &\lim_{r\rightarrow\infty}\, \int_{-\infty}^{\infty} d \lambda \frac{h(\lambda)}{\vert c(2\lambda)\vert^2} \left\{ c(2\lambda) {\overline{c(2\mu)}} +{\overline{c(2\lambda)}} c(2\mu)\right\}\, \frac{\sin(\lambda-\mu)r}{\lambda-\mu} \\ \nonumber
 &=& \lim_{r\rightarrow\infty}\, \int_{-\infty}^{\infty} d \lambda \left[h(\lambda)\left\{ \frac{{\overline{c(2\mu)}}}{\overline{c(2\lambda)}} + \frac{c(2\mu)}{ c(2\lambda)}\right\}\right]\, \frac{\sin(\lambda-\mu)r}{\lambda-\mu}\\ \nonumber
 &=& \pi\, h(\mu) \left\{ \frac{{\overline{c(2\mu)}}}{\overline{c(2\mu)}} + \frac{c(2\mu)}{ c(2\mu)}\right\} = 2 \pi h(\mu),
 \end{eqnarray} from Remark \ref{bothellone}, since the function in the above parenthesis $[\dots]$, denoted by $q= q(\lambda)$ and its classical Fourier transform are shown to be in $L^1$.  In fact, set for a fixed $\mu$
 \begin{eqnarray}
 \hspace{4mm}q(\lambda) = h(\lambda)\left\{ \frac{{\overline{c(2\mu)}}}{\overline{c(2\lambda)}} + \frac{c(2\mu)}{ c(2\lambda)}\right\} = h(\lambda)\left\{ \frac{c(-2\mu)}{c(-2\lambda)} + \frac{c(2\mu)}{ c(2\lambda)}\right\},\hspace{2mm}\lambda\in{\Bbb R}.
  \end{eqnarray} It is shown similarly as in the proof of Lemma \ref{elonefinite} that $q$ belongs to $L^1$. In the following we will show  ${\widehat{q}}^{cl}\in L^1$.  
  The function $h(\lambda)$, the factor of $q(\lambda)$ is smooth and rapidly decreasing, since $h\in{\mathcal{PW}}{\Bbb C}_{even}$ is the image of the classical Fourier transform of a rapidly decreasing smooth function. Then, $h$ satisfies that for any $N\in{\Bbb N}$ there exists a constant $C_N > 0$
  \begin{eqnarray*}
  \left| \left(\frac{d}{d\lambda}\right)^k h(\lambda) \right| \leq C_N (1+\vert\lambda\vert)^{-N},\, k=0,1,2.
  \end{eqnarray*} On the other hand, for the functions $c(\pm 2\lambda)^{-1}$ it is observed  
that there exist $K_k>0$ and a positive integer $N_0$ such that    
  $\displaystyle{\left\vert \left(\frac{d}{d\lambda}\right)^k c(\pm 2\lambda)^{-1}\right\vert}\leq K_k(1+\vert\lambda\vert)^{N_0}$, $k=0,1,2$. 
\,  In fact, using the exact form of $c(\pm 2\lambda)$ which appears at Lemma \ref{c2lambda} in terms of $\Gamma(z)$, we have by the aids of Digamma function $\psi(z) := d/dz \log \Gamma(z)$ and its derivative $\psi'(z)$ that 
  \begin{eqnarray*}\left(\frac{d}{d\lambda}\right)
 c(2\lambda)^{-1} &=& i\left\{\psi(a + i\lambda) + \psi(b +  i\lambda)- 2\psi(2i\lambda)\right\}\,c( 2\lambda)^{-1}, \\ \nonumber
 \left(\frac{d}{d\lambda}\right)^2
 c(2\lambda)^{-1} &=& -\big[
 \left\{\psi'(a +i\lambda) + \psi'(b +  i\lambda)- 4\psi'(2 i\lambda)\right\} \\ \nonumber
 &+& \left\{\psi(a + i\lambda) + \psi(b +  i\lambda)- 4\psi(2 i\lambda)\right\}^2
 \big]\, c(2\lambda)^{-1}, \hspace{2mm}\lambda\in{\Bbb R},
   \end{eqnarray*}where $a = Q/2$, $b = (n-Q)/2$.
   By using  formulae \cite[6.3.5, 6.4.6]{AbSt} and \cite[6.3.18, 6.4.12]{AbSt} of $\psi(z)$ together with (\ref{estimatec-2}) we can see that $c(2\lambda)^{-1}$ and similarly $c(-2\lambda)^{-1}$ are of $C^2$, even at $\lambda=0$ and $\displaystyle{\left\vert \left(\frac{d}{d\lambda}\right)^k c(\pm 2\lambda)^{-1}\right\vert}$, $k=1,2$  are bounded from above by $K_k(1+ \vert \lambda\vert)^{N_0}$. Therefore, the function $q(\lambda)$
    is of $C^2$ and $\displaystyle{\left(\frac{d}{d\lambda}\right)^kq}$ belongs to $L^1$, $k=0,1,2$ and then from the degree decreasing property of the classical Fourier transform there exists a constant $C>0$ such that $\vert {\widehat{q}}^{cl}(\xi)\vert \leq C(1+\vert\xi\vert)^{-2} $, from which ${\widehat{q}}^{cl}$ belongs to $L^1$. Thus, by applying Theorem \ref{riemannlebesgue} together with Remark \ref{bothellone} we get (\ref{hmu}).

 \vspace{2mm}Similarly we have
 \begin{eqnarray}& &\lim_{r\rightarrow\infty}\, \int_{-\infty}^{\infty} d \lambda \frac{h(\lambda)}{\vert c(2\lambda)\vert^2} \big\{ c(2\lambda) c(2\mu) + {\overline{c(2\lambda)}}{\overline{c(2\lambda)}} \big\}\, \frac{\sin(\lambda+\mu)r}{\lambda+\mu} \\ \nonumber
 &=& 2 \pi\, h(-\mu).
 \end{eqnarray}
 
 As a consequence of (\ref{9.18}), (\ref{finalformula}) we have for a fixed $\mu$, since $h=h(\lambda)$ is even \begin{eqnarray}\label{integralformula}& &
\lim_{r\rightarrow\infty} \int_{-\infty}^{\infty} d \lambda\, \frac{h(\lambda)}{\vert c(2\lambda)\vert^2}\, \int_{B(o;r)} \varphi_{\lambda}(r(x)){\overline \varphi}_{\mu}(r(x)) dv_{B(o;r)} \\ \nonumber
&=&\omega_{n-1}\,\frac{\Gamma(n/2)^2}{ \pi}\hspace{0.5mm}C_g\hspace{0.5mm}\frac{1}{4}\hspace{0.5mm}4\hspace{0.5mm}\pi\hspace{0.5mm} h(\mu) = \hspace{0.5mm}\omega_{n-1}\, \Gamma(n/2)^2\hspace{0.5mm}C_g\hspace{0.5mm}h(\mu).
\end{eqnarray}

 \begin{proposition}\label{integral-2} Let $h = h(\lambda)$\, $\in {\mathcal{PW}}{\Bbb C}_{even}$. Then for any real $\mu$ 
\begin{multline}\label{formulax}
h(\mu)=\lim_{r\rightarrow\infty} \int_{-\infty}^{\infty} d \lambda \frac{h(\lambda)}{\omega_{n-1}\, \Gamma(n/2)^2\hspace{0.5mm}C_g\hspace{0.5mm} \vert c(2\lambda)\vert^2}\\
\times \int_{x\in B(o; r)} \varphi_{\lambda}(r(x)){\overline \varphi}_{\mu}(r(x)) dv_{B(o; r)} 
.
\end{multline}
\end{proposition}

\begin{lemma} \label{c2lambda}
\begin{eqnarray}\label{c2lambda-eq}
c(2\lambda) =\hspace{0.5mm}2\hspace{0.5mm}\sqrt{\pi}\, \frac{1}{\Gamma(n/2)}\, {\bf c}(\lambda).
\end{eqnarray}
Here, ${\bf c}(\lambda)$ is the Harish-Chandra ${\bf c}$-function, given at (\ref{harishChandrac}).
\end{lemma}

The above formula (\ref{c2lambda-eq}) is derived as follows;
\begin{eqnarray}\label{cfunction}
c(2\lambda) 
&=& \frac{2^Q\, \Gamma(i\lambda) \Gamma(\frac{1}{2}+ i\lambda)}{\Gamma(\frac{n-Q}{2} + i\lambda) \Gamma( \frac{Q}{2}  + i \lambda)} \\ \nonumber
&=& \sqrt{2\pi}\, 2^{Q-\left(2i\lambda-1/2\right)} \frac{\Gamma(2i\lambda)}{\Gamma(Q/2+i\lambda) \Gamma((n-Q)/2 + i\lambda)}\\ \nonumber
&=& 2\hspace{0.5mm}\sqrt{\pi}\, 2^{Q-2i\lambda} \frac{\Gamma(2i\lambda)}{\Gamma(Q/2+i\lambda) \Gamma((n-Q)/2 + i\lambda)},
\end{eqnarray}
in which we make use of the duplicative formula for $\Gamma(z)$ \cite[6.1.18]{AbSt};
\begin{eqnarray*}\Gamma(2z) = \frac{1}{\sqrt{2\pi}} \, 2^{(2z-\frac{1}{2})}\, \Gamma(z)\, \Gamma(\frac{1}{2}+z),\hspace{2mm}z\in{\Bbb C}.
\end{eqnarray*}

From Lemma \ref{c2lambda} $\displaystyle{\vert c(2\lambda)\vert^2 = \frac{4\pi}{(\Gamma(n/2))^2}\hspace{1mm} \vert {\bf c}(\lambda)\vert^2}$ so 
\begin{eqnarray}\omega_{n-1} \Gamma(n/2)^2 C_g  \vert c(2\lambda)\vert^2 &=& \omega_{n-1} \Gamma(n/2)^2 C_g\cdot \frac{4\pi}{\Gamma(n/2)^2}\vert {\bf c}(\lambda)\vert^2 \\ \nonumber
&=& 4 \pi\, \omega_{n-1}\, C_g \vert {\bf c}(\lambda)\vert^2 
\end{eqnarray}
Thus, we have
\begin{multline}\label{formulax-2}
h(\mu)=\lim_{r\rightarrow\infty} \int_{-\infty}^{\infty} d \lambda \frac{h(\lambda)}{4 \pi \omega_{n-1}\hspace{0.5mm}C_g\hspace{0.5mm} \vert {\bf c}(\lambda)\vert^2}\\
\times \int_{x\in B(o; r)} \varphi_{\lambda}(r(x)){\overline \varphi}_{\mu}(r(x)) dv_{B(o; r)}.
\end{multline} Since the integration over $B(o;r)$ commutes with the integration with respect to $\lambda$, 
\begin{multline}\label{formulax-3}
h(\mu)=\lim_{r\rightarrow\infty} \int_{x\in B(o; r)} {\overline \varphi}_{\mu}(r(x)) dv_{B(o; r)}\left( \int_{-\infty}^{\infty} d \lambda \frac{h(\lambda)\varphi_{\lambda}(r(x))}{4 \pi \omega_{n-1}\hspace{0.5mm}C_g\hspace{0.5mm} \vert {\bf c}(\lambda)\vert^2}
 \right) 
.
\end{multline}  The spherical functions $\varphi_{\lambda}(r)$ are real for real $\lambda$. We may therefore write 
\begin{multline}\label{formulax-3}
h(\mu)=\lim_{r\rightarrow\infty} \int_{x\in B(o; r)} \varphi_{\mu}(r(x)) dv_{B(o; r)}\left( \int_{-\infty}^{\infty} d \lambda \frac{h(\lambda){\overline{\varphi}}_{\lambda}(r(x))}{4 \pi \omega_{n-1}\hspace{0.5mm}C_g\hspace{0.5mm} \vert {\bf c}(\lambda)\vert^2}
 \right) 
.\end{multline}
\vspace{2mm}
Next we will identify the constant appeared in the argument with the constant $1/d$ in (\ref{inversionformula}).
  We have  $\omega_{n-1} = 2 \pi^{n/2}/\Gamma(n/2)$ from Note \ref{areaunitsphere} and  $C_g= 2^{-2Q} k_g$ from (\ref{c}) so that
  \begin{eqnarray}
4\pi\, \omega_{n-1}\hspace{0.5mm}C_g
= 4 \pi\, \left(2 \frac{\pi^{n/2}}{\Gamma(n/2)}\right) \times \left(2^{-2Q} k_g\right) 
= \frac{2^{3-2Q}}{\Gamma(n/2)}\, \pi^{n/2+1} k_g
\end{eqnarray}which gives the constant $1/d$. 
\begin{remark}\label{dc0} Equality $\displaystyle{d = \frac{c_0}{2}}$ holds for the constants $d$ and $c_0$, when $(X,g)$ is Damek-Ricci. Here $c_0$ is defined at (\ref{damekriccic0}) and one has
  \begin{eqnarray}
\frac{c_0}{2} = \frac{1}{2}\cdot 2^{k-2}\, \pi^{-(n/2\, +1)}\Gamma(n/2).
\end{eqnarray}On the other hand for an $n$-dimensional $(X,g)$ of  hypergeometric type
\begin{eqnarray}d = \frac{1}{4\pi \omega_{n-1} C_g} = \frac{2^{2Q-3}}{k_g}\, \pi^{-(n/2 + 1)}\, \Gamma(n/2).
\end{eqnarray}
If $(X,g)$ is  Damek-Ricci, then $k_g = 2^{m+k}$ ($m = \dim {\mathfrak v}$ and $k = \dim{\mathfrak z}$) as shown in Remark \ref{damekriccikg}. Thus
$\displaystyle{d = 2^{(2Q-3-m-k)}\,\pi^{-(n/2 + 1)}
  \Gamma(n/2)}$. 
  Here $Q = m/2 + k$ so $2Q-3-m-k = k-3$ and hence $d = c_0/2$. 
\end{remark}

 \section{The Spherical Fourier transform and the inversion formula}
 
Let $(X,g)$ be a harmonic Hadamard manifold of hypergeometric type having $Q > 0$. Then, the spherical Fourier transform is defined in (\ref{sphfourier}) of section \ref{intro} by
 \begin{equation*}
\begin{split}
{\mathcal H}f(\lambda)
=& \int_X f(x) \varphi_{\lambda}(x) dv_X = \omega_{n-1} \int_0^{\infty} f(r) \varphi_{\lambda}(r) \Theta(r) dr \\
=& \hspace{0.5mm}\frac{2\hspace{0.5mm}\pi^{n/2}}{\Gamma(\frac{n}{2})} k_g \int_0^{\infty} dr\, f(r)\hspace{0.5mm}\left(\cosh \frac{r}{2}\right)^{2 Q} \left(\tanh \frac{r}{2}\right)^{n-1}\hspace{0.5mm}\varphi_{\lambda}(r)
\end{split}
\end{equation*} 
 for smooth radial functions $f = f(x)$ with compact support on $X$,  identified with functions $f = f(r)$ of geodesic distance $r = d(x,o)$ to the reference point $o$. Here $\displaystyle{k_g = -\ \frac{2^n}{3 Q -(n-1)}\hspace{0.5mm}{\rm Ric}_g }$ is a constant which  depends upon $(X,g)$.

  The inversion formula for the spherical Fourier transform (see Theorem \ref{inversionthm}) takes the form;
 \begin{equation*}
   f(r) = {d}\ \int_{-\infty}^{\infty} \frac{d \lambda}{\vert{\bf c}(\lambda)\vert^2} {\mathcal H}f(\lambda) {\overline \varphi}_{\lambda}(r),
 \end{equation*}
 where $\displaystyle{{d} = \frac{1}{4\pi \omega_{n-1} C_g} = k_g^{-1}\, 2^{2Q-3} \pi^{-(n/2+1)} \Gamma(n/2)}$ and ${\bf c}(\lambda) $ is the function of $\lambda\in {\Bbb R}$, known as Harish-Chandra $c$-function, given in (\ref{harishChandrac}).

\begin{proof}[Proof of Theorem \ref{inversionthm}]
Let $f = f(r)$ be a smooth function of compact support with respect to $r \geq 0$.
Put  $h(\lambda) = {\mathcal H} f(\lambda)$. Then, we have from (\ref{formulax-3}) the equality
  \begin{equation}
    ({\mathcal H}{\widetilde h})(\lambda)= h(\lambda)
  \end{equation}
  by setting $\displaystyle{{\widetilde h}(r) :=  {d} \int_{-\infty}^{\infty} \frac{d \lambda}{\vert{\bf c}(\lambda)\vert^2} h(\lambda) {\overline \varphi}_{\lambda}(r) }$. Then, the following lemma tells us that ${\widetilde h}(r)$ belongs to ${\mathcal{C}}_c^{\infty}(X)^{rad}$.
 Since the right hand side of above is just $h ={\mathcal H} f$, we have  from the injectivity of ${\mathcal H}$,
\begin{equation*}
{\widetilde h}(r) = f(r),
\end{equation*}
from which the inversion formula is obtained. Here the injectivity of ${\mathcal H}$ is from \cite[Theorem 3.12]{PS}. 
 \end{proof} 
 \begin{lemma}\label{hormandertrick} \rm For $h \in {\mathcal{PW}}{\Bbb C}_{even}$  
 the function ${\widetilde h}= {\widetilde h}(r)$ defined in the proof of Theorem \ref{inversionthm} belongs to ${\mathcal{C}}_c^{\infty}(X)^{rad}$. 
 
 \end{lemma} This lemma is verified in \cite{ItohSatohpre}. The idea for proving the support compactness of ${\widetilde{h}}$ is the H\"ormander's trick (\cite[Chap. I, proof of Theorem 1.7.7]{Hormander}, \cite[sect. 4]{Flensted}). Here, we give an outline of its proof.
 \begin{proof}We may assume $h \in {\mathcal{PW}}{\Bbb C}_{even}^R$ for some $R>0$.  
 We show that the integration is well defined and then  ${\rm supp}\,{\widetilde h} \subset [0,R]$ and finally that ${\widetilde h}$ is  smooth.
 Since, $h$ is of exponential type and that $\varphi_{\lambda}(r)$ and $\vert c(2\lambda)\vert$ for $\lambda\in {\Bbb R}$ are estimated as in Lemmata  \ref{estimationofc}, \ref{estimationofspherical}, respectively, the integrand 
 $\displaystyle{\frac{1}{\vert{\bf c}(\lambda)\vert^2} h(\lambda) {\overline \varphi}_{\lambda}(r) }$ is integrable.  To see  ${\rm supp}\,{\widetilde h} \subset [0,R]$ we write 
 \begin{eqnarray}\label{integral-2}
 \int_0^{\infty} h(\lambda) \varphi_{\lambda}(r) \frac{d\lambda}{\vert {\bf c}(\lambda)\vert^2}
  = \frac{2\sqrt{\pi}}{\Gamma(\frac{n}{2})} \int_{-\infty}^{\infty} h(\lambda) \frac{\Phi_{2\lambda}(\frac{r}{2})}{c(-2\lambda)} d\lambda
 \end{eqnarray} by making use of (\ref{linearcombination}) and Lemma \ref{c2lambda}. 
 It suffices to show that there exists a constant $K > 0$ such that for any fixed $\eta>0$
 \begin{eqnarray}\label{integral-3}
 \vert {\widetilde h}(r)\vert \leq K e^{(R-r)\eta}.
 \end{eqnarray} Then it is easily seen that ${\widetilde h}(r)=0$ for $r>R$.
 To obtain (\ref{integral-3}) from (\ref{integral-2}) we make use of Cauchy's integral theorem.  It is shown that the right hand integral of (\ref{integral-2}) coincides with the line integral along  the line ${\Pi}_{\eta}$ : ${\Bbb R}\ni \xi \mapsto \lambda(\xi) = \xi + i\eta$
 \begin{eqnarray}
\frac{2\sqrt{\pi}}{\Gamma(\frac{n}{2})} \int_{{\Pi}_{\eta}} h(\lambda) \frac{\Phi_{2\lambda}(\frac{r}{2})}{c(-2\lambda)} d\lambda
  \end{eqnarray} with respect to any fixed $\eta > 0$.  Here the functions $h(\lambda)$, $\Phi_{2\lambda}(r/2)$  are holomorphic in the upper half plane $U= \{\lambda=\xi+i\eta\, \vert\, \eta \geq 0\}$. 
  $c(-2\lambda)^{-1}$ is also holomorphic in $U$. In fact, this assertion is obtained as follows. In fact,
  \begin{eqnarray}
c(-2\lambda)^{-1} = (2\hspace{0.5mm}\sqrt{\pi}\, 2^{Q-2i\lambda})^{-1} \frac{\Gamma(Q/2-i\lambda) \Gamma((n-Q)/2 - i\lambda)}{\Gamma(-2i\lambda)},
\end{eqnarray}
  is a meromorphic function with respect to $\lambda\in{\Bbb C}$ whose poles are just poles of the numerator. Since the poles of $\Gamma(z)$ are given by  $\{0, -1,-2,\cdots\}$, the poles of $c(-\lambda)^{-1}$ are located in $\{\lambda= \xi+ i\eta \, \vert\,  \eta \leq -c_0\}$, where $c_0= \min(Q, n-Q)$ is positive from $(n-1)/2\leq Q \leq n-1$ appeared in Remark \ref{volumeentropy}.  
  Therefore we can apply the Cauchy's integral theorem, since asymptotic decay of $h(\lambda)$, $\Phi_{2\lambda}(t)$ and $c(-2\lambda)^{-1}$, $\lambda = \xi + i \eta$, as $\xi \rightarrow +\infty$ for a fixed $\eta > 0$ are well estimated.  Refer to \cite[Corollary 9]{Flensted} for the decay of $c(-2\lambda)^{-1}$. By using the estimation of $\Phi_{\mu}(t)$ given in \cite[Theorem 2]{Flensted}, we have the following. Take an arbitrary integer $N$. Then, there exists a constant $K > 0$ depending on $N$ for which it holds for any fixed $t=r/2 > 0$ and any fixed $\eta > 0$ 
 \begin{eqnarray}\label{estimate}
 \left\vert h(\xi+i\eta) \frac{\Phi_{2(\xi+i\eta)}(t)}{c(-2(\xi+i\eta))}\right\vert \leq K e^{(R-2t)\eta} (1+ \vert\xi+i\eta\vert)^{(\frac{n-1}{2}-N)}.
 \end{eqnarray} Choose $N$ as $N >  (n-1)/2 + 2$ so we obtain (\ref{integral-3}).

  Smoothness of ${\widetilde h}$ stems from the following. In fact, for any integer $m$ there exists a constant $K_m > 0$ such that
  \begin{eqnarray*}
  \int_0^{\infty} \big\vert h(\lambda) \frac{d^m}{dr^m} \varphi_{\lambda}(r) \big\vert \frac{d\lambda}{\vert c(2\lambda)\vert^2} \leq K_m \int_0^{\infty} \vert h(\lambda) (1+ \lambda)^{m+n-1}\vert d\lambda < +\infty
   \end{eqnarray*} holds for $r\in[0,\infty)$. This is shown  by applying  \cite[Theorem 2, (i) (ia)]{Flensted} together with  Lemma \ref{estimationofc}.
 \end{proof}

\subsection*{Acknowledgements}
The authors would like to thank the referee for valuable comments.
The second author is supported in part by JSPS Grants-in-Aid for Scientific (B) 15K17545.


\begin{thebibliography}{SKK2}

\bibitem{AbSt} M. Abramowitz and I.A. Stegun(eds.), {\it Pocketbook of Mathematical Functions, Abridged edition of Handbook of Mathematical Functions}, Verlag, Harri Deutsch, Thun, 1984.

\bibitem{ADY} J.-P. Anker, E. Damek and C. Yacoub, Spherical Analysis on Harmonic $AN$ Groups, Ann. Scuola Norm. Sup. Pisa Cl. Sci. \textbf{23} (1996), no.4, 643-679.


\bibitem{BTV} J. Berndt, F. Tricerri and L. Vanhecke, {\it Generalized Heisenberg Groups and Damek-Ricci Harmonic Spaces}, Lecture Notes in Math. \textbf{1598}, Springer-Verlag, Berlin, 1995.

\bibitem{Besse}A. L. Besse, {\it Manifolds all of whose Geodesics are Closed}, Springer-Verlag, Berlin, 1978.

\bibitem{BCG} G. Besson, G. Courtois and S. Gallot, Entropies et Rigidit$\acute{\rm e}$s des Espaces Localement Sym$\acute{\rm e}$triques de Courbure Strictement N$\acute{\rm e}$gative, Geom. Funct. Anal. \textbf{5} (1995), 731-799.

\bibitem{DR} E. Damek and F. Ricci, A class of nonsymmetric harmonic Riemannian spaces, Bull. Amer. Math. Soc. \textbf{27} (1992), 139-142.

\bibitem{DR2} E. Damek and F. Ricci, Harmonic analysis on solvable extensions of $H$-type groups, J. Geom. Anal. \textbf{2} (1992), 213-248.

\bibitem{DiBlas} B. Di Blasio, Paley-Wiener type theorems on harmonic extensions of $H$-type groups, Monat. Math., {\bf 123} (1997), 21-42.

\bibitem{Do} M. P. Do Carmo, {\it Riemannian Geometry}, Birkh\"auser, Boston, 1992.

\bibitem{Erdelyi} A. Erd\'elyi, W. Magnus, F. Oberhettinger and F. Tricomi, {\it Higher Transcendental Functions}, vol. I, Robert E. Krieger Publ., Malabar, 1981.

\bibitem{Flensted} M. Flensted-Jensen, Paley-Wiener type theorems for a differential operator connected with symmetric spaces, Ark. Mat., {\bf 10} (1972), 143-162.
\bibitem{Flensted-2} M. Flensted-Jensen, Spherical Functions on a simply connected semisimple Lie group, Math. Ann., {\bf 228} (1977),65-92.
\bibitem{GHL} S. Gallot, D. Hulin and J. Lafontaine, {\it Riemannian Geometry}, second edition, Springer-Verlag, Berlin, 1990.

\bibitem{G} F. G\"otze, Verallgemeinerung einer Integraltransformation von Mehler-Fock durch den von Kuipers und Meulenbeld Eigenf\"uhrten Kern $P_k^{m,n}(z)$, Indag. Math. \textbf{27} (1965), 396-404.

\bibitem{GrVan}A. Gray and L. Vanhecke, Riemannian geometry as determined by the volumes of small geodesic balls, Acta Math. \textbf{11} (1979), 157-198.

\bibitem{H}S. Helgason, {\it Groups and Geometric Analysis}, Acad. Press,  Orlando, 1984.
\bibitem{Hormander} L. H\"ormander, {\it Linear Partial Differential Operators}, Springer-Verlag, Berlin, 1969.

\bibitem{IKPS}M. Itoh, S. Kim, J. Park and H. Satoh, Harmonic Hadamard manifolds of prescribed Ricci curvature and volume entropy, Kyushu J. Math. \textbf{70} (2016), 267-280.
\bibitem{ItohSatohpre}M. Itoh, H. Satoh, 
Spherical Fourier Transform on harmonic manifolds of hypergeometric type and Plancherel Theorem, in preparation.
\bibitem{ISS}M. Itoh, H. Satoh and Y. J. Suh, Horospheres and Hyperbolicity of
Hadamard manifolds, Differential Geom. Appl. \textbf{35} (2014), suppl., 50-68. 

\bibitem{K}
G. Knieper, New results on noncompact harmonic manifolds,  Comment. Math. Helv. \textbf{87} (2012), 669-703.


\bibitem{K2016}
G. Knieper, A survey on noncompact harmonic and asymptotically harmonic manifolds, 146-197, in \textit{Geometry, topology, and dynamics in negative curvature}, London Math. Soc. Lecture Note Ser. \textbf{425}, Cambridge Univ. Press, Cambridge, 2016.

\bibitem{Ko-x} T.H. Koornwinder, A new proof of a Paley-Wiener type theorem for the Jacobi transform, Ark. Mat., {\bf 13}(1975), 145-159.

\bibitem{Ko} T. H. Koornwinder, Jacobi Functions and Analysis on Noncompact Semisimple Lie Groups, 1-85, in  {\it Special Functions:Group Theoretical Aspects and Applications}, Reidel Publ., Dordrecht, 1984.

\bibitem{L}F. Ledrappier, Harmonic measures and Bowen-Margulis measures, Israel J. Math. \textbf{71} (1990), 275-287.

\bibitem{Lich}A. Lichnerowicz, Sur les espaces Riemanniens compl\`etement harmoniques, Bull. Soc. Math. France \textbf{72} (1944), 146-168.




\bibitem{N}
Y. Nikolayevsky, Two theorems on harmonic manifolds,
Comment. Math. Helv. \textbf{80} (2005), 29-50.


\bibitem{PS} N. Peyerimhoff and E. Samiou, Integral Geometric Properties of Non-compact Harmonic Spaces, J. Geom. Anal. \textbf{25} (2013), 122-148.

\bibitem{RS} A. Ranjan and H. Shah, Harmonic Manifolds with Minimal Horospheres, J. Geom. Anal. \textbf{12} (2002), 683-694.

\bibitem{Ri} F. Ricci, The spherical transform on harmonic extensions of $H$-type groups, Rend. Sem. Mat. Univ. Politec. Torino \textbf{50} (1992), 381-392.

\bibitem{R} F. Rouvi\'{e}re, Espaces de Damek-Ricci, Geometrie et Analyse, Seminaires et Congres \textbf{7}, 45-100, Soc. Math. France, Paris, 2003.

\bibitem{Sz}Z. Szab$\acute{\rm o}$, The Lichnerowicz Conjecture on Harmonic Manifolds, J. Differential Geom. \textbf{31} (1990), 1-28.

\bibitem{WW}E. T. Whittaker and G. N. Watson, {\it A Course of Modern Analysis}, fourth edit., Cambridge Univ. Press, London, 1973.

\end{thebibliography}
\end{document}